\documentclass{amsart}

\RequirePackage{fix-cm}
\usepackage{graphicx}
\usepackage{amsmath, amsopn,amstext,amscd,amsfonts,amssymb}
\usepackage{dsfont}
\usepackage{comment}
\usepackage[active]{srcltx}
\usepackage{graphicx, epsfig, subfig}

\usepackage{textcomp}

\usepackage{algorithm}

\makeatletter
\def\BState{\State\hskip-\ALG@thistlm}
\makeatother

\def\downbar#1{
\setbox10=\hbox{$#1$}
            \dimen10=\ht10 \advance\dimen10 by 2.5pt
            \ifdim \dimen10<15pt 
               \advance\dimen10 by -0.5pt
               \dimen11=\dimen10
               \advance\dimen10 by 2.5pt
               \lower \dimen11
            \else \lower \ht10 \fi
            \hbox {\hskip 1.5pt \vrule height \dimen10 depth \dp10}}
\def\upbar#1{
\setbox10=\hbox{$#1$}
            \dimen10=\ht10 \advance\dimen10 by \dp10 \advance\dimen10 by 2.5pt
            \ifdim \dimen10<15pt 
                \advance\dimen10 by 2pt \fi
            \raise 2.5pt \hbox {\hskip -1.5pt \vrule height \dimen10}}

\usepackage{multicol}
\usepackage{colortbl}
\usepackage{exscale}
\usepackage{amssymb,latexsym,amsthm,amsfonts,color,fancyhdr,mathrsfs}
\usepackage{lscape}
\usepackage{rotating}
\usepackage{caption}

\newtheorem{definition}{\bf Definition}[section]
\newtheorem{theorem}{\bf Theorem}[section]
\newtheorem{proposition}{\bf Proposition}[section]
\newtheorem{lemma}{\bf Lemma}[section]

\newtheorem{corollary}{\bf Corollary}[section]
\newtheorem{remark}{\bf Remark}[section]

\newtheorem*{thmA}{\bf Theorem A}
\newtheorem*{thmB}{\bf Theorem B}
\newtheorem*{thmC}{\bf Theorem C}

\numberwithin{equation}{section}
\begin{document}
\title[On the functional equation for classical OPS]{On the functional equation for classical orthogonal polynomials on lattices}

\author{K. Castillo}
\address{University of Coimbra, CMUC, Dep. Mathematics, 3001-501 Coimbra, Portugal}
\email{kenier@mat.uc.pt}
\author{D. Mbouna}
\address{University of Coimbra, CMUC, Dep. Mathematics, 3001-501 Coimbra, Portugal}
\email{mbouna@mat.uc.pt}
\author{J. Petronilho}
\address{University of Coimbra, CMUC, Dep. Mathematics, 3001-501 Coimbra, Portugal}
\email{josep@mat.uc.pt}

\subjclass[2010]{42C05, 33C45}
\date{\today}
\keywords{Functional equation, regular functional, classical orthogonal polynomials, lattices, Racah polynomials, Askey-Wilson polynomials}

\begin{abstract}
Necessary and sufficient conditions for the regularity 
of solutions of the functional equation
appearing in the theory of classical orthogonal polynomials on lattices are stated.
Moreover, the functional Rodrigues formula and a closed formula for the recurrence coefficients are presented.
\end{abstract}
\maketitle

\section{Introduction}\label{introduction}

Let $\mathcal{P}$ be the vector space of all polynomials with complex coefficients
and let $\mathcal{P}^*$ be its algebraic dual.
$\mathcal{P}_n$ denotes the space of all polynomials with degree less than or equal to $n$.
Define $\mathcal{P}_{-1}:=\{0\}$.
A simple set in $\mathcal{P}$ is a sequence $(P_n)_{n\geq0}$ such that
$P_n\in\mathcal{P}_n\setminus\mathcal{P}_{n-1}$ for each $n$.
A simple set $(P_n)_{n\geq0}$ is called an orthogonal polynomial sequence (OPS)
with respect to ${\bf u}\in\mathcal{P}^*$ if 
$$
\langle{\bf u},P_nP_k\rangle=h_n\delta_{n,k}\quad(n,k=0,1,\ldots;\;h_n\in\mathbb{C}\setminus\{0\})\;,
$$
where $\langle{\bf u},f\rangle$ is the action of ${\bf u}$ on $f\in\mathcal{P}$. 
${\bf u}$ is called regular, or quasi-definite, if
there exists an OPS with respect to it.
It is well known that ${\bf u}$ is regular if and only if $\det\big[u_{i+j}\big]_{i,j=0}^n\neq0$ for each $n=0,1,\ldots$, 
where $u_n:=\langle{\bf u},z^n\rangle$. 
In 1940 Ya. L. Geronimus proved the following result \cite[Theorem II]{G1940}: 
\vspace{2mm}
\begin{changemargin}{0.8cm}{0.8cm} 
{\em A sequence of polynomials and the sequence of its derivatives are orthogonal with respect to sequences of moments\footnote{
Recall that given a sequence of complex numbers $(u_n)_{n\geq0}$, $(P_n)_{n\geq0}$ is said to be orthogonal with respect to $(u_n)_{n\geq0}$ if it is an OPS with respect to the functional ${\bf u}\in\mathcal{P}^*$ defined by 
$$
\langle {\bf u},f\rangle:=\sum_{k=0}^na_{k}u_k\;,\quad f(z)=\sum_{k=0}^na_{k}z^k\in\mathcal{P}\,.
$$}
 $(u_n)_{n\geq0}$ and $(v_n)_{n\geq0}$, respectively, if and only if 
\begin{equation}\label{A1}
(na+d)u_{n+1}+(nb+e)u_n+ncu_{n-1}=0 
\end{equation}
for each $n$, where $a,b,c,d,e$ are complex numbers such that
$na+d\neq0$ and $\det\big[u_{i+j}\big]_{i,j=0}^n\neq0$, and} 
\begin{equation}\label{A3}
v_n=au_{n+2}+bu_{n+1}+cu_{n}\,.
\end{equation}
\end{changemargin}
\vspace{2mm}

Given ${\bf u}\in\mathcal{P}^*$ and $g\in\mathcal{P}$, the derivative of ${\bf u}$ and the left multiplication of ${\bf u}$ by $g$  are the functionals ${\bf D}{\bf u}\in\mathcal{P}^*$ and $g{\bf u}\in\mathcal{P}^*$ defined by
$$
\langle {\bf D}{\bf u},f\rangle:=-\langle{\bf u}, f'\rangle \;,\quad
\langle g{\bf u},f\rangle:=\langle{\bf u}, fg\rangle\quad(f\in\mathcal{P})\;.
$$
This allows us to rewrite the difference equation \eqref{A1} as a functional equation\footnote{Usually 
\eqref{GPa} is called Pearson (or Pearson-type) functional (or distributional) equation.
Without wanting to get into nitpicking about the nature of the name, 
it seems appropriate to call \eqref{GPa} Geronimus-Pearson (functional) equation.}
\begin{equation}\label{GPa}
{\bf D}(\phi{\bf u})=\psi{\bf u}\;,
\end{equation}
where $\phi(z)=az^2+bz+c$ and $\psi(z)=dz+e$, and 
\eqref{A3} as ${\bf v}=\phi{\bf u}$, being ${\bf u}$ and ${\bf v}$ the linear functionals on $\mathcal{P}$ given by 
$\langle{\bf u},z^n\rangle=u_n$ and $\langle{\bf v},z^n\rangle=v_n$.
It is worth mentioning that Geronimus' result was motivated by Hahn's characterization of the classical OPS\footnote{As pointed out in \cite{ARS1995}, this characterization was indeed stated before by Sonine \cite{S1887}.} of Hermite, Laguerre, Jacobi, and Bessel ---the only OPS such that their sequences of derivatives are also OPS \cite{H1935}.  
We call a functional {\it ${\bf u}\in\mathcal{P}^*$ classical in Hahn's sense}, or, simply, $\mathrm{D}-${\it classical}, if the corresponding OPS is such that the associated sequence of derivatives is also an OPS. 
Accordingly, Geronimus' theorem can be rewritten as follows\footnote{As an application of this result, Geronimus also proved that a monic OPS $(P_n)_{n\geq0}$ is $\mathrm{D}-$classical if  $$P_n=\mbox{$\frac{1}{n+1}$}P_{n+1}'+b_nP_n'+c_nP_{n-1}'\quad (n=0,1,\ldots),\quad P_{-1}=0\,,$$  where $b_n$ and $c_n$ are complex numbers (see  \cite[(42)]{G1940}). Of course, the converse of this sentence was known long before.}:

\begin{thmA}
${\bf u}\in\mathcal{P}^*$ is $\mathrm{D}-$classical if and only if there exist $\phi\in\mathcal{P}_2$ and $\psi\in\mathcal{P}_1$ such that ${\bf u}$ satisfies \eqref{GPa} and the conditions 
\begin{equation}\label{A2a}
na+d\neq0\;,\quad H_n:=\det\big[u_{i+j}\big]_{i,j=0}^n\neq0 \quad (n=0,1,2,\ldots)
\end{equation}
hold,
where $a:=\phi^{\prime\prime}/2$, $d:=\psi'$, and $u_n:=\langle{\bf u},z^n\rangle$.
\end{thmA}

\noindent
The condition $na+d\neq0$ for every $n$ means that $(\phi,\psi)$ is an admissible pair \cite{M1991}.
Note that, concerning the applicability of Theorem A, although it is trivial to check the admissibility condition for a given pair $(\phi,\psi)\in\mathcal{P}_2\times\mathcal{P}_1$, the same does not holds (in general) for the regularity condition $H_n\neq0$, because the order of $H_n$ grows with $n$. Theorem B in bellow  improves Theorem A, showing that conditions \eqref{A2a} may be replaced by rather simple ones ---see \eqref{Aregular} in bellow---, as well as how to compute the monic OPS $(P_n)_{n\geq0}$ with respect to ${\bf u}$ using only the pair $(\phi,\psi)$. This is achieved from the explicit formulas for the coefficients $B_n$ and $C_n$ appearing in the three-term recurrence relation (TTRR) for $(P_n)_{n\geq0}$, namely
\begin{equation}\label{ATTRRa}
zP_n(z)=P_{n+1}(z)+B_nP_n(z)+C_nP_{n-1}(z)\quad (n=0,1,\ldots)\,,\quad P_{-1}(z)=0\,.
\end{equation}

\begin{thmB}
${\bf u}\in\mathcal{P}^*\backslash\{\bf 0\}$ is $\mathrm{D}-$classical if and only if there exist polynomials 
$\phi(z):=az^2+bz+c$ and $\psi(z):=dz+e$ such that \eqref{GPa} holds and 
\begin{equation}\label{Aregular}
na+d\neq0\;,\quad \phi\left(-\frac{nb+e}{2na+d}\right)\neq0 \quad (n=0,1,\ldots).
\end{equation}
Under such conditions, the monic OPS $(P_n)_{n\geq0}$ with respect to ${\bf u}$ fulfils \eqref{ATTRRa} with
\begin{equation}\label{ATTRRAb}
B_n=\frac{ne_{n-1}}{d_{2n-2}}-\frac{(n+1)e_{n}}{d_{2n}}\, ,\;
C_{n+1}=-\frac{(n+1)d_{n-1}}{d_{2n-1}d_{2n+1}}\phi\Big(-\frac{e_n}{d_{2n}}\Big)
\; (n=0,1,\ldots)
\end{equation}
where $d_n:=na+d$ and $e_n:=nb+e$. Moreover, the following functional Rodrigues formula holds:
$$
P_n{\bf u}=k_n\,{\bf D}^n\big(\phi^n{\bf u}\big)\;,\quad k_n:=\prod_{j=0}^{n-1}d_{n+j-1}^{-1}\quad(n=0,1,\ldots)\;.
$$
\end{thmB}

Theorem B was stated in \cite{MP1994}.\footnote{The explicit formulas \eqref{ATTRRAb} appeared earlier in Suslov's article \cite{S1989}, where they have been derived by a different method, although the regularity conditions for ${\bf u}$ had not been discussed therein.}
Observe that conditions \eqref{Aregular} mean that  $(\phi,\psi)$ is an admissible pair and $\psi+n\phi'\nmid\phi$ for each $n=0,1,\ldots$.
It is well known that if there exist ${\bf u}\in\mathcal{P}^*$ and nonzero polynomials $\phi\in\mathcal{P}_2$ and $\psi\in\mathcal{P}_1$ such that \eqref{GPa} holds, then the regularity of ${\bf u}$ implies that $(\phi,\psi)$ is an admissible pair (see e.g. \cite{M1991,MP1994}). Therefore, the above definition of $\mathrm{D}-$classical functional is equivalent to the following statement: {\it ${\bf u}\in\mathcal{P}^*\backslash\{{\bf 0}\}$ is $\mathrm{D}-$classical if and only if it is regular and there exist nonzero polynomials $\phi\in\mathcal{P}_2$ and $\psi\in\mathcal{P}_1$ such that ${\bf u}$ satisfies \eqref{GPa}}. This statement is indeed the definition of $\mathrm{D}-$classical functional adopted nowadays within the algebraic theory of orthogonal polynomials, developed by Maroni \cite{M1985,M1988,M1991} (see also \cite{P2017}).  

Recently, an analogue of Theorem B in the framework of OPS with respect to the discrete Hahn operator has been proved in \cite{RKDP2020}.
In order to give here the statement of this analogue, we need to recall several definitions. Given complex numbers $q$ and $\omega$ subject to the conditions 
\begin{equation}\label{q-notexp}
|q-1|+|\omega|\neq0\;,\quad
q\not\in\big\{ 0,{\rm e}^{2ij\pi/n}\;|\;1\leq j\leq n-1\;;\;\;n=2,3,\ldots\big\},
\end{equation}
the (ordinary) Hahn operator $D_{q,\omega}:\mathcal{P}\to\mathcal{P}$, studied by W. Hahn in \cite{H1949}, is
$$
D_{q,\omega}f(x):=\frac{f(qx+\omega)-f(x)}{(q-1)x+\omega}\quad (f\in\mathcal{P})\;.
$$
$D_{q,\omega}$ induces an operator on the dual space,
${\bf D}_{q,\omega}:\mathcal{P}^*\to\mathcal{P}^*$, given by (see \cite{F1998})
$$
\langle{\bf D}_{q,\omega}{\bf u},f\rangle:=
-q^{-1}\langle{\bf u},D_{q,\omega}^*f\rangle\quad({\bf u}\in\mathcal{P}^*\;,\;f\in\mathcal{P})\;,
$$
where $D_{q,\omega}^*:=D_{1/q,-\omega/q}$.
These definitions are on the basis of the following notion of classical OPS:  {\it ${\bf u}\in\mathcal{P}^*$ is $(q,\omega)-$classical if it is regular and there exist nonzero polynomials $\phi\in\mathcal{P}_2$ and $\psi\in\mathcal{P}_1$, such that} 
\begin{equation}\label{GPaqw}
{\bf D}_{q,\omega}(\phi{\bf u})=\psi{\bf u}\;.
\end{equation}
Define also operators
$L_{q,\omega}:\mathcal{P}\to\mathcal{P}$ and ${\bf L}_{q,\omega}:\mathcal{P}^*\to\mathcal{P}^*$ by (see \cite{F1998})
$$
L_{q,\omega}f(x):=f(qx+\omega)\;,\quad
\langle {\bf L}_{q,\omega}{\bf u},f\rangle:=q^{-1}\langle {\bf u},L_{q,\omega}^*f\rangle
\quad \big(f\in\mathcal{P}\;,\;{\bf u}\in\mathcal{P}^*\big)\;,
$$
where $L_{q,\omega}^*:=L_{1/q,-\omega/q}$. Finally, recall the definition of the $q-$bracket:
$$
[\alpha]_q:=
\left\{\begin{array}{cl}
\displaystyle\frac{q^\alpha-1}{q-1}\;,&\mbox{\rm if}\quad q\neq1 \\ [0.75em]
\alpha\;,&\mbox{\rm if}\quad q=1
\end{array}
\right.\quad(\alpha,q\in\mathbb{C})\;.
$$
The $(q,\omega)-$analogue of Theorem B, stated in \cite[Theorem 1.2]{RKDP2020}, reads as follows: 

\begin{thmC}
Fix $q,\omega\in\mathbb{C}$ fulfilling $(\ref{q-notexp})$. 
${\bf u}\in\mathcal{P}^*\backslash\{\bf 0\}$ is $(q,\omega)-$classical if and only if there exist polynomials 
$\phi(z):=az^2+bz+c$ and $\psi(z):=dz+e$ such that ${\bf u}$ satisfies the functional equation \eqref{GPaqw} and the conditions
$$
d_n\neq0\,,\quad \phi\Big(-\frac{e_n}{d_{2n}}\Big)\neq0\quad (n=0,1,2,\ldots)
$$
hold, where $d_n\equiv d_n(q):=dq^n+a[n]_{q}$ and $e_n\equiv e_n(q,\omega):=eq^n+(\omega d_n+b)[n]_q$. 
Under these conditions, the monic OPS $(P_n)_{n\geq0}\equiv(P_n(\cdot;q,\omega))_{n\geq0}$
with respect to ${\bf u}$ satisfies the TTRR \eqref{ATTRRa}, where
$$
B_n=\omega[n]_q+\frac{[n]_qe_{n-1}}{d_{2n-2}}-\frac{[n+1]_qe_{n}}{d_{2n}}\, ,\quad
C_{n+1}:=-\frac{q^{n}[n+1]_qd_{n-1}}{d_{2n-1}d_{2n+1}}\phi\Big(-\frac{e_n}{d_{2n}}\Big)
$$
$(n=0,1,\ldots)$. 
In addition, the Rodrigues formula
$$
P_n{\bf u}=k_n\,{\bf D}_{1/q,-\omega/q}^n\Big(\Phi(\cdot;n){\bf L}_{q,\omega}^n{\bf u}\Big)
\quad(n=0,1,\ldots)
$$
holds in $\mathcal{P}^*$, where
$$
k_n:=q^{n(n-3)/2}\prod_{j=0}^{n-1}d_{n+j-1}^{-1}\;,\quad
\Phi(x;n):=\prod_{j=1}^n\phi\big(q^jx+\omega[j]_q\big)\;.
$$
\end{thmC}

At this stage, a natural question arises, asking for analogues of theorems B and C for OPS on lattices
in the sense of Nikiforov, Suslov, and Uvarov \cite{NSU1991} 
(see also \cite{ARS1995,I2005,KLS2010,R2014}).
The aim of this work is answering to this question.
The structure of the paper is the following.
In Section \ref{remarksNUL} we review some basic definitions and results focusing on the theory of OPS on lattices.
Section \ref{SecRodrigues} contains a functional Rodrigues formula for classical OPS on lattices.
In Section \ref{S-main} we state our main results, giving the analogues of theorems B and C for OPS on lattices 
(for linear, $q-$linear, quadratic, and $q-$quadratic lattices).
Finally, as a straightforward application of the main results, in Section \ref{secAW} we revisit the Racah polynomials and the Askey-Wilson polynomials,
computing their recurrence coefficients directly from the functional equation fulfilled by the associated regular functional.

\section{Preliminary results on lattices}\label{remarksNUL}

In this section we review the definition of lattice. In addition, we derive a Leibniz formula (on lattices) for the left multiplication of a
functional by a polynomial, as well as some preliminary results needed.

\subsection{Definitions and basic properties}

A lattice is a mapping $x(s)$ given by
\begin{equation}
\label{xs-def}
x(s):=\left\{
\begin{array}{ccl}
\mathfrak{c}_1 q^{-s} +\mathfrak{c}_2 q^s +\mathfrak{c}_3 & {\rm if} &  q\neq1 \;,\\ [0.75em]
\mathfrak{c}_4 s^2 + \mathfrak{c}_5 s +\mathfrak{c}_6 & {\rm if} &  q =1
\end{array}
\right.
\end{equation}
($s\in\mathbb{C}$), where $q>0$ (fixed) and $\mathfrak{c}_j$ ($1\leq j\leq6$) are (complex) constants,
that may depend on $q$, such that $(\mathfrak{c}_1,\mathfrak{c}_2)\neq(0,0)$ if $q\neq1$,
and $(\mathfrak{c}_4,\mathfrak{c}_5,\mathfrak{c}_6)\neq(0,0,0)$ if $q=1$.
In the case $q=1$, the lattice is called quadratic if $\mathfrak{c}_4\neq0$,
and linear if $\mathfrak{c}_4=0$;
and in the case $q\neq1$, it is called $q-$quadratic if $\mathfrak{c}_1\mathfrak{c}_2\neq0$,
and $q-$linear if $\mathfrak{c}_1\mathfrak{c}_2=0$ (cf. \cite{ARS1995}).
Notice that
$$
\frac{x\big(s+\frac12\big)+x\big(s-\frac12\big)}{2}=\alpha x(s)+\beta\;,
$$
where $\alpha$ and $\beta$ are given by
\begin{equation}\label{alpha-beta}
\alpha:=\frac{q^{1/2}+q^{-1/2}}{2}\;,\quad
\beta:=\left\{
\begin{array}{ccl}
(1-\alpha)\mathfrak{c}_3 & {\rm if} &  q\neq1 \;,\\ [0.75em]
\mathfrak{c}_4/4 & {\rm if} &  q =1\;.
\end{array}
\right.
\end{equation}
The lattice $x(s)$ fulfills (cf. \cite{ARS1995}): 
\begin{align*}
\frac{x(s+n)+x(s)}{2}&=\alpha_n x_n (s) +\beta_n \,, \\
x(s+n)-x(s)&=\gamma_n \nabla x_{n+1} (s)
\end{align*}
($n=0,1,\ldots$), where $x_{\mu}(s):=x\big(s+\mbox{$\frac\mu2$}\big)$, $\nabla f(s):=f(s)-f(s-1)$,
and $(\alpha_n)_{n\geq0}$, $(\beta_n)_{n\geq0}$, and $(\gamma_n)_{n\geq0}$ are sequences of numbers
generated by the following system of difference equations
\begin{align}
&\alpha_0 =1\;,\quad \alpha_1=\alpha\;,\quad\alpha_{n+1} -2\alpha\alpha_n +\alpha_{n-1} =0 \label{1.2}\\
&\beta_0 =0\;,\quad \beta_1 =\beta\;,\quad\beta_{n+1} -2\beta_n + \beta_{n-1} =2\beta\alpha_{n} \label{1.2b} \\ 
&\gamma_0 =0\;,\quad \gamma_1 =1\;,\quad \gamma_{n+1} -\gamma_{n-1} =2\alpha_n \label{1.2c}
\end{align}
($n=1,2,\ldots$). The explicit solutions of these difference equations are
\begin{align}
& \alpha_n = \frac{q^{n/2} +q^{-n/2}}{2}\,, \label{alpha-n} \\
& \beta_n = \displaystyle\left\{
\begin{array}{ccl}
\displaystyle\beta\,\left(\frac{q^{n/4}-q^{-n/4}}{q^{1/4}-q^{-1/4}}\right)^2 & \mbox{\rm if} & q\neq1 \\ [1em]
\beta\,n^2 & \mbox{\rm if} & q=1\,,
\end{array}
\right. \label{beta-n} \\
& \gamma_n = \displaystyle\left\{
\begin{array}{ccl}
\displaystyle\frac{q^{n/2}-q^{-n/2}}{q^{1/2}-q^{-1/2}} & \mbox{\rm if} & q\neq1 \\ [1em]
n & \mbox{\rm if} & q=1\;.
\end{array}
\right. \label{gamma-n}
\end{align}
These formulas may be easily checked (alternatively, see \cite{ARS1995}).
We point out the following relations:
\begin{align}
&\gamma_{n+1}-2\alpha\gamma_n +\gamma_{n-1} =0\;, \label{gamma-n-bis} \\
&\alpha_{n}+\gamma_{n-1}=\alpha\gamma_{n}\;, \label{gamma-n-bis1} \\
&(2\alpha^2-1)\alpha_{n}+(\alpha^2-1)\gamma_{n-1}=\alpha\alpha_{n+1} \label{gamma-n-bis2}\;, \\
&\gamma_{2n}=2\alpha_{n}\gamma_{n} \label{form-1}\;, \\
&\alpha_{n}^2+(\alpha^2-1)\gamma_{n}^2=\alpha_{2n}=2\alpha_n^2-1 \label{form-2}\;, \\
&\alpha_{n-1}-\alpha\alpha_n=(1-\alpha^2)\gamma_{n} \label{form-3}\;,\\
&\alpha +\alpha_n\gamma_n=\alpha_{n-1}\gamma_{n+1} \label{form-4}\;, \\
&1 +\alpha_{n+1}\gamma_n=\alpha_{n}\gamma_{n+1} \label{form-5}
\end{align}
($n=0,1,\ldots$), with the conventions $\alpha_{-1}:=\alpha$ and $\gamma_{-1}:=-1$,
consistently with \eqref{alpha-n} and \eqref{gamma-n}.

\begin{definition}\label{def-DxSxNUL} 
Let $x(s)$ be a lattice given by \eqref{xs-def}.  
The $x-$derivative operator on $\mathcal{P}$, $\mathrm{D}_x$, 
and the $x-$average operator on $\mathcal{P}$, $\mathrm{S}_x$, 
are the operators on $\mathcal{P}$ defined for each $f\in\mathcal{P}$ so that $\deg(\mathrm{D}_x f)=\deg f-1$, $\deg(\mathrm{S}_x f)=\deg f$, and
\begin{align}
\mathrm{D}_x f(x(s))&= \frac{f\big(x(s+\frac{1}{2})\big)-f\big(x(s-\frac{1}{2})\big)}{x(s+\frac{1}{2})-x(s-\frac{1}{2})}\;, \label{AWxsG}  \\
\mathrm{S}_x f(x(s))&= \frac{f\left(x\big(s+\frac{1}{2}\big)\right) + f\left(x\big(s-\frac{1}{2}\big)\right)}{2}\,.
\label{AverOperG}
\end{align}
\end{definition}

The relations \eqref{AWxsG} and \eqref{AverOperG} appear in \cite[(3.2.4)-(3.2.5)]{NSU1991} up to a shift in the variable $s$. 
The operators  $\mathrm{D}_x$ and $\mathrm{S}_x$ on $\mathcal{P}$ induce two operators on the dual space $\mathcal{P}^*$, namely
$\mathbf{D}_x:\mathcal{P}^*\to\mathcal{P}^*$ and $\mathbf{S}_x:\mathcal{P}^*\to\mathcal{P}^*$, via the following definition (cf. \cite{FK-NM2011}):

\begin{definition} 
Let $x(s)$ be a lattice given by \eqref{xs-def}. 
For each ${\bf u}\in\mathcal{P}^*$, the functionals $\mathbf{D}_x{\bf u}\in\mathcal{P}^*$
and $\mathbf{S}_x{\bf u}\in\mathcal{P}^*$ are defined by
\begin{equation}\label{PNUL-def-Dxu-Sxu}
\langle \mathbf{D}_x{\bf u},f\rangle:=-\langle {\bf u},\mathrm{D}_x f\rangle\; ,\quad
\langle \mathbf{S}_x{\bf u},f\rangle:=\langle {\bf u},\mathrm{S}_x f\rangle\quad (f\in\mathcal{P})\,.
\end{equation}
We call $\mathbf{D}_x{\bf u}$ the $x-$derivative of ${\bf u}$ and $\mathbf{S}_x{\bf u}$ the $x-$average of ${\bf u}$.
\end{definition}
Hereafter, $z:=x(s)$ being a lattice given by (\ref{xs-def}), we consider two fundamental polynomials, 
$\texttt{U}_1$ and $\texttt{U}_2$, introduced in \cite{FK-NM2011}, defined by
\begin{align}
\texttt{U}_1(z) &:=(\alpha^2 -1)z+\beta(\alpha +1)\;,\quad  \label{U1x} \\
\texttt{U}_2(z) &:=(\alpha^2 -1)z^2+2\beta(\alpha +1)z+\delta\;, \label{U2x}
\end{align}
where $\delta\equiv\delta_x$ is a constant with respect to the lattice, given by\footnote{The constant $\delta_x$ appears in \cite{FK-NM2011} without an explicit expression; a full expression, different from the ones provided here, is given in \cite{MFN-S2017}.}
\begin{equation} \label{1.5}
\delta := \left(\frac{x(0)+x(1)-2\beta(\alpha+1)}{2\alpha}\right)^2-x(0)x(1)\;.
\end{equation}
A straightforward computation shows that
\begin{equation} \label{delta-simples}
\delta =\left\{
\begin{array}{lcl}
(\alpha^2-1)\big(\mathfrak{c}_3^2-4\mathfrak{c}_1\mathfrak{c}_2\big) & \mbox{\rm if} & q\neq1 \;, \\ [0.5em]
\frac14\mathfrak{c}_5^2-\mathfrak{c}_4\mathfrak{c}_6& \mbox{\rm if} & q=1 
\end{array}
\right.
\end{equation}
and
\begin{equation} \label{U1-simples-bis}
\texttt{U}_1 (z) =\left\{
\begin{array}{lcl}
(\alpha^2-1)\big(z-\mathfrak{c}_3\big) & \mbox{\rm if} & q\neq1 \;, \\ [0.5em]
\frac{1 }{2}\mathfrak{c}_4 & \mbox{\rm if} & q=1 \,.
\end{array}
\right.
\end{equation}
Hence
\begin{equation} \label{U2-simples-bis}
\texttt{U}_2(z) =\left\{
\begin{array}{lcl}
(\alpha^2-1)\big((z-\mathfrak{c}_3)^2-4\mathfrak{c}_1\mathfrak{c}_2\big) & \mbox{\rm if} & q\neq1 \;, \\ [0.5em]
\mathfrak{c}_4(z-\mathfrak{c}_6)+\frac14\mathfrak{c}_5^2& \mbox{\rm if} & q=1 \;.
\end{array}
\right.
\end{equation}
Finally, we recall the following useful relations which can be easily proved (see \cite{FK-NM2011}):
\begin{align}
&\mathrm{D}_x\texttt{U}_1=\alpha^2-1 \;,\quad
& \mathrm{S}_x\texttt{U}_1&=\alpha\texttt{U}_1 \;,  \label{DxSxU1} \\
&\mathrm{D}_x\texttt{U}_2=2\alpha\texttt{U}_1 \;,\quad
& \mathrm{S}_x\texttt{U}_2&=\alpha^2\texttt{U}_2+\texttt{U}_1^2 \;. \label{DxSxU2}
\end{align}

\subsection{Properties of the $x-$derivative and $x-$average operators}
We start by pointing out some useful properties. 

\begin{lemma}\label{propertyDxSx}
Let $f,g\in\mathcal{P}$ and ${\bf u}\in\mathcal{P}^*$. Then the following properties hold:\rm
\begin{align}
 &\mathrm{D}_x \big(fg\big)= \big(\mathrm{D}_x f\big)\big(\mathrm{S}_x g\big)+\big(\mathrm{S}_x f\big)\big(\mathrm{D}_x g\big)\;, \label{1.3} \\
 &\mathrm{S}_x \big( fg\big)= \big(\mathrm{D}_x f\big) \big(\mathrm{D}_x g\big)\texttt{U}_2 +\big(\mathrm{S}_x f\big) \big(\mathrm{S}_x g\big)\;, \label{1.4} \\
&\mathrm{S}_x\mathrm{D}_x f= \alpha\mathrm{D}_x\mathrm{S}_x f-\mathrm{D}_x\big(\texttt{U}_1\mathrm{D}_x f\big) \;, \label{1.4aa} \\
&\mathrm{S}_x^2 f = \alpha^{-1}\mathrm{S}_x \big(\texttt{U}_1 \mathrm{D}_x f  \big) +\alpha^{-1}\texttt{U}_2 \mathrm{D}_x^2 f +f\;, \label{c} \\
 &f\mathrm{S}_xg
=\mathrm{S}_x\Big(\big(\mathrm{S}_xf-\alpha^{-1}\texttt{U}_1\,\mathrm{D}_xf\big)g\Big)
-\alpha^{-1}\texttt{U}_2\mathrm{D}_x\big(g\mathrm{D}_xf\big)\;, \label{PropNUL2} \\
 &f\mathrm{D}_x g =\mathrm{D}_x\Big(\big(\mathrm{S}_xf-\alpha^{-1}\texttt{U}_1\,\mathrm{D}_xf\big)g\Big)
-\alpha^{-1}\mathrm{S}_x\big(g\mathrm{D}_xf\big)\;, \label{PropNUL1} \\
&\mathbf{D}_x(f{\bf u})
=\left(\mathrm{S}_x f -\alpha^{-1}\texttt{U}_1\mathrm{D}_x f\right)\mathbf{D}_x {\bf u} +
\alpha^{-1} \big(\mathrm{D}_x f\big) \mathbf{S}_x {\bf u}\;, \label{a} \\
&\mathbf{S}_x (f {\bf u})
= \left(\mathrm{S}_x f -\alpha^{-1}\texttt{U}_1\mathrm{D}_x f\right)\mathbf{S}_x{\bf u}+\alpha^{-1}\big(\mathrm{D}_xf\big) \mathbf{D}_x (\texttt{U}_2 {\bf u})\;, \label{b} \\
&\mathbf{S}_x (f {\bf u}) 
= \big(\alpha \texttt{U}_2 -\alpha^{-1}\texttt{U}_1^2  \big)(\mathrm{D}_x f)\mathbf{D}_x {\bf u} +
\big(\mathrm{S}_x f +\alpha^{-1}\texttt{U}_1\mathrm{D}_x f  \big)\mathbf{S}_x {\bf u}\;,  \label{1} \\
&\mathbf{D}_x^2 \big(\texttt{U}_2 {\bf u}\big)
=\alpha \mathbf{S}_x^2 {\bf u} +\mathbf{D}_x \big(\texttt{U}_1 \mathbf{S}_x {\bf u}  \big) -\alpha {\bf u}\;,  \label{2aa} \\
&\mathbf{D}_x^2 \big(\texttt{U}_2 {\bf u}\big)
= (2\alpha-\alpha^{-1}) \mathbf{S}_x^2 {\bf u}
+\alpha^{-1}\texttt{U}_1 \mathbf{D}_x \mathbf{S}_x {\bf u} -\alpha {\bf u}\;, \label{2bb} \\
&\mathbf{D}_x \mathbf{S}_x {\bf u} 
= \alpha\mathbf{S}_x \mathbf{D}_x {\bf u}+\mathbf{D}_x\big(\texttt{U}_1\mathbf{D}_x{\bf u}\big)\;. \label{3aa}
\end{align}
\end{lemma}
\begin{proof}
The reader may encounter properties (\ref{1.3})--(\ref{b}) in 
\cite[Propositions 5--7]{MFN-S2017}.
To prove \eqref{1}, set $f=\texttt{U}_2$ in (\ref{a}) and then use \eqref{DxSxU2} to obtain
$$
\mathbf{D}_x (\texttt{U}_2 {\bf u})
=\big(\alpha^2\texttt{U}_2-\texttt{U}_1^2\big)\mathbf{D}_x{\bf u}+2\texttt{U}_1\mathbf{S}_x{\bf u}\;.
$$
Replacing this expression in the right-hand side of \eqref{b} we obtain \eqref{1}.
Next, taking arbitrarily $f\in\mathcal{P}$, we have
\begin{align*}
\langle\mathbf{D}_x^2\big(\texttt{U}_2{\bf u}\big),f\rangle &=\langle{\bf u},\texttt{U}_2\mathrm{D}_x^2f\rangle
 =\langle{\bf u},\alpha \mathrm{S}_x^2f-\mathrm{S}_x(\texttt{U}_1\mathrm{D}_xf)-\alpha f\rangle \\
 &=\langle \alpha \mathbf{S}_x^2 {\bf u} +\mathbf{D}_x \big(\texttt{U}_1 \mathbf{S}_x {\bf u}  \big) -\alpha {\bf u},f \rangle\;,
\end{align*}
where the second equality holds by (\ref{c}). This proves (\ref{2aa}).
Setting $f=\texttt{U}_1$ in (\ref{a}) and replacing therein ${\bf u}$ by $\mathbf{S}_x{\bf u}$,
and taking into account (\ref{DxSxU1}),
we deduce
$$
\mathbf{D}_x\big(\texttt{U}_1\mathbf{S}_x{\bf u}\big)
=\alpha^{-1}\texttt{U}_1\mathbf{D}_x\mathbf{S}_x {\bf u} +(\alpha-\alpha^{-1})\mathbf{S}_x^2 {\bf u}\;.
$$
Substituting this into the right-hand side of (\ref{2aa}) we obtain (\ref{2bb}).
Finally, (\ref{3aa}) follows easily from (\ref{PNUL-def-Dxu-Sxu}) and (\ref{1.4aa}).
\end{proof}
\begin{proposition}
For the lattice $\;x(s)=\mathfrak{c}_1q^{-s}+\mathfrak{c}_2q^s+\mathfrak{c}_3$, the following holds:
\begin{align}
\mathrm{D}_x z^n &=\gamma_n z^{n-1}+u_nz^{n-2}+v_nz^{n-3}+\cdots\;, \label{Dx-xn}\\
\mathrm{S}_x z^n &=\alpha_n z^n+\widehat{u}_nz^{n-1}+\widehat{v}_nz^{n-2}+\cdots \label{Sx-xn}
\end{align}
$(n=0,1,\ldots)$, where $\alpha_n$ and $\gamma_n$ are given by \eqref{alpha-n} and \eqref{gamma-n}, and
\begin{align}
u_n &:=\big(n\gamma_{n-1}-(n-1)\gamma_n\big)\mathfrak{c}_3\;, \label{unD} \\
v_n &:=\big(n\gamma_{n-2}-(n-2)\gamma_n\big)\mathfrak{c}_1\mathfrak{c}_2 \label{vnD} \\
&\qquad+\mbox{$\frac12$}\big(n(n-1)\gamma_{n-2}-2n(n-2)\gamma_{n-1}+(n-1)(n-2)\gamma_{n}\big)\mathfrak{c}_3^2 \;, \nonumber\\
\widehat{u}_n &:=n(\alpha_{n-1}-\alpha_n)\mathfrak{c}_3\;, \label{unS} \\
\widehat{v}_n &:=n(\alpha_{n-2}-\alpha_n)\mathfrak{c}_1\mathfrak{c}_2
+n(n-1)(\alpha-1)\alpha_{n-1}\mathfrak{c}_3^2\;. \label{vnS}
\end{align}
\end{proposition}
\begin{proof}
The proof is given by mathematical induction on $n$. 
For $n=0$, we have $\mathrm{D}_x z^0 =\mathrm{D}_x 1=0$ and $\mathrm{S}_x z^0 =1$. 
Since $\gamma_0 =0$ and $\alpha_0= 1$, 
we see that (\ref{Dx-xn})--(\ref{Sx-xn}) hold for $n=0$. 
Next suppose that relations (\ref{Dx-xn})--(\ref{Sx-xn}) are true for all integer numbers less than or equal to a fixed nonnegative integer number $n$ (induction hypothesis). 
Using this hypothesis together with (\ref{1.3})--(\ref{1.4}), we obtain
\begin{align*}
\mathrm{D}_x z^{n+1} &=\mathrm{D}_x z^n ~ \mathrm{S}_x z + \mathrm{S}_x z^n ~\mathrm{D}_x z 
=  (\alpha z+\beta)\mathrm{D}_x z^n + \mathrm{S}_x z^n  \\
&= (\alpha_n +\alpha \gamma_n)z^n +(\alpha u_n +\widehat{u}_n +\beta \gamma_n)z^{n-1} +(\alpha v_n +\widehat{v}_n +\beta u_n)z^{n-2} +\cdots.
\end{align*}
Similarly,
\begin{align*}
\mathrm{S}_x z^{n+1} &=\texttt{U}_2 (z)\mathrm{D}_x z~\mathrm{D}_x z^n  + \mathrm{S}_x z^n ~\mathrm{S}_x z 
=  \texttt{U}_2 (z)\mathrm{D}_x z^n + (\alpha z+\beta)\mathrm{S}_x z^n  \\
&= \left(\alpha \alpha_n +(\alpha ^2 -1)\gamma_n \right) z^{n+1} 
+\left((\alpha ^2 -1)\left(u_n -2\gamma_n \mathfrak{c}_3 \right) +\alpha \widehat{u}_n +\beta \alpha_n  \right) z^{n} \\
&\quad+\left( (\alpha ^2 -1) \left(v_n -2u_n \mathfrak{c}_3 +(\mathfrak{c}_3^2 -4 \mathfrak{c}_1\mathfrak{c}_2 ) \gamma_n \right) +\alpha \widehat{v}_n +\beta \widehat{u}_n \right) z^{n-1} +\cdots .
\end{align*}
Therefore, using relations (\ref{gamma-n-bis})--(\ref{form-5}), we obtain 
(\ref{Dx-xn})--(\ref{Sx-xn}) with $n$ replaced by $n+1$. 
Consequently, (\ref{Dx-xn})--(\ref{Sx-xn})  holds for each nonnegative integer $n$.
\end{proof}

\subsection{A Leibniz formula}

Here we state a useful version of the Leibniz formula, involving the $x-$derivative operator,
for the left multiplication of a functional by a polynomial.
We need to state some preliminary results.
\begin{lemma}
Let $f \in \mathcal{P}$. Then
\begin{align}\label{DxnSxf}
\mathrm{D}_x ^n \mathrm{S}_x f = \alpha_n \mathrm{S}_x \mathrm{D}_x ^n f +\gamma_n \mbox{\rm $\texttt{U}_1$} \mathrm{D}_x ^{n+1}f \quad (n=0,1,\ldots).
\end{align}
\end{lemma}
\begin{proof}
Once again, we use mathematical induction on $n$. Clearly, (\ref{DxnSxf}) holds for $n=0$. Suppose that (\ref{DxnSxf}) is true for all positive integers less than or equal to a fixed integer $n$. Then by using successively (\ref{DxnSxf}) firstly for $n$ (induction hypothesis) and secondly for $n=1$ (already proved) with $f$ replaced by $\mathrm{D}_x ^n f$, and applying (\ref{1.3}) to $\mathrm{D}_x (\texttt{U}_1 \mathrm{D}_x ^{n+1} f)$, and also taking into account \eqref{DxSxU1}, we deduce
\begin{align*}
\mathrm{D}_x ^{n+1} \mathrm{S}_x f &=\mathrm{D}_x \left(\mathrm{D}_x ^n \mathrm{S}_x f \right)=\mathrm{D}_x \left( \alpha_n \mathrm{S}_x \mathrm{D}_x ^n f +\gamma_n \texttt{U}_1 \mathrm{D}_x ^{n+1}f \right) \\
&=\alpha_n \mathrm{D}_x \mathrm{S}_x (\mathrm{D}_x ^n f) +\gamma_n\mathrm{D}_x \left( \texttt{U}_1 \mathrm{D}_x ^{n+1}f\right) \\
&=\alpha_n \left(\alpha \mathrm{S}_x \mathrm{D}_x ^{n+1} f +\texttt{U}_1 \mathrm{D}_x ^{n+2}f  \right) +\gamma_n \left((\alpha ^2 -1)\mathrm{S}_x \mathrm{D}_x ^{n+1}f +\alpha \texttt{U}_1 \mathrm{D}_x ^{n+2}f  \right) \\
&= \left(\alpha \alpha_n +(\alpha ^2 -1)\gamma_n  \right)\mathrm{S}_x \mathrm{D}_x ^{n+1}f +(\alpha_n +\alpha \gamma_n)\texttt{U}_1 \mathrm{D}_x ^{n+2}f.
\end{align*}
Finally, using properties  \eqref{form-3}, \eqref{1.2}, \eqref{gamma-n-bis1}, and \eqref{gamma-n-bis}, we 
see that \eqref{DxnSxf} is true whenever $n$ is replaced by $n+1$. Hence \eqref{DxnSxf} is true for all $n$.
\end{proof}

 The next result is a functional version of \eqref{DxnSxf}.
 
\begin{lemma}\label{property3}
Let ${\bf u}\in\mathcal{P}^*$. Then
\begin{equation}\label{33}
\alpha \mathbf{D}_x ^n \mathbf{S}_x {\bf u}
= \alpha_{n+1} \mathbf{S}_x \mathbf{D}_x^n {\bf u}
+\gamma_n\mbox{\rm $\texttt{U}_1$}\mathbf{D}_x^{n+1}{\bf u}\quad (n=0,1,2,\ldots).
\end{equation}

\end{lemma}
\begin{proof}
We prove \eqref{33} by mathematical induction on $n$.
Since $\alpha_1=\alpha$ and $\gamma_0=0$, then \eqref{33} is trivial for $n=0$.
For $n=1$, \eqref{33} is obtained multiplying both sides of \eqref{3aa} by $\alpha$ and
taking into account that, by \eqref{a} and \eqref{DxSxU1}, the equality
$\alpha\mathbf{D}_x\big(\texttt{U}_1\mathbf{D}_x{\bf u}\big)
=\texttt{U}_1\mathbf{D}_x^2{\bf u}+(\alpha^2-1)\mathbf{S}_x\mathbf{D}_x{\bf u}$ holds,
and recalling also that $\alpha_2=2\alpha^2-1$ and $\gamma_1=1$.
Suppose now that property \eqref{33} holds for a fixed integer $n\in\mathbb{N}$ (induction hypothesis).
Then, we have
\begin{equation}\label{3Aa1}
\alpha \mathbf{D}_x^{n+1} \mathbf{S}_x {\bf u}
=\mathbf{D}_x \big(\alpha \mathbf{D}_x^n \mathbf{S}_x {\bf u}\big)
=\alpha_{n+1} \mathbf{D}_x \mathbf{S}_x \mathbf{D}_x^n {\bf u}
+\gamma_n\mathbf{D}_x\big(\texttt{U}_1\mathbf{D}_x^{n+1}{\bf u}\big).
\end{equation}
Considering \eqref{33} for $n=1$ and replacing therein
${\bf u}$ by $\mathbf{D}_x^{n}{\bf u}$, we obtain
\begin{equation}\label{3Aa2}
\mathbf{D}_x\mathbf{S}_x\mathbf{D}_x^n{\bf u}
=\alpha^{-1}\alpha_2\mathbf{S}_x\mathbf{D}_x^{n+1}{\bf u}
+\alpha^{-1}\gamma_n\texttt{U}_1\mathbf{D}_x^{n+2}{\bf u}\;.
\end{equation}
Moreover, using again \eqref{a} and \eqref{DxSxU1}, we deduce
\begin{equation}\label{3Aa3}
\mathbf{D}_x\big(\texttt{U}_1\mathbf{D}_x^{n+1}{\bf u}\big)
=\alpha^{-1}\texttt{U}_1\mathbf{D}_x^{n+2}{\bf u}
+\alpha^{-1}(\alpha^2-1)\mathbf{S}_x\mathbf{D}_x^{n+1}{\bf u}\;.
\end{equation}
Putting \eqref{3Aa2} and \eqref{3Aa3} into the right-hand side of \eqref{3Aa1}
and taking into account \eqref{gamma-n-bis1} and \eqref{gamma-n-bis2},
we obtain \eqref{33} with $n$ replaced by $n+1$.
This proves \eqref{33}.
\end{proof}

 Next, we introduce an operator $\mathrm{T}_{n,k}:\mathcal{P}\to\mathcal{P}$
($n=0,1,\ldots;\, k=0,1,\ldots, n$),
defined for each $f\in\mathcal{P}$ as follows: if $n=k=0$, set
\begin{align}\label{T00}
\mathrm{T}_{0,0}f &:=f\;;
\end{align}
and if $n\geq1$ and $0\leq k\leq n$, define recurrently
\begin{align}\label{Tnk}
\mathrm{T}_{n,k}f&:= \mathrm{S}_x \mathrm{T}_{n-1,k}f
-\frac{\gamma_{n-k}}{ \alpha_{n-k}}\texttt{U}_1 \mathrm{D}_x \mathrm{T}_{n-1,k}f
+\frac{1}{\alpha_{n+1-k}} \mathrm{D}_x \mathrm{T}_{n-1,k-1}f\,,
\end{align}
with the conventions $\mathrm{T}_{n,k}f:=0$ whenever $k>n$ or $k<0$. 
Note that
$$
\deg \mathrm{T}_{n,k}f\leq\deg f-k\;.
$$
We are ready to state the following

\begin{proposition}[Leibniz's formula]\label{Leibniz-rule-NUL}
Let ${\bf u}\in\mathcal{P}^*$ and $f\in\mathcal{P}$. Then
\begin{align}
\mathbf{D}_x ^n \big(f{\bf u}\big)
=\sum_{k=0}^{n} \mathrm{T}_{n,k}f\, \mathbf{D}_x^{n-k} \mathbf{S}_x^k {\bf u}
 \quad (n=0,1,\ldots),\label{leibnizfor-NUL}
\end{align}
where $\mathrm{T}_{n,k}f$ is a polynomial defined by \eqref{T00}--\eqref{Tnk}.
\end{proposition}
\begin{proof}
The proof is done by mathematical induction on $n$.
Clearly, (\ref{leibnizfor-NUL}) is true if $n=0$.
Suppose now that (\ref{leibnizfor-NUL}) holds for a fixed nonnegative integer $n$.
Then
\begin{equation}\label{eq-Leib1}
\mathbf{D}^{n+1}_x \big(f{\bf u}\big)
=\mathbf{D}_x\big(\mathbf{D}^{n}_x (f{\bf u})\big)
=\sum_{k=0}^n \mathbf{D}_x\big(\mathrm{T}_{n,k}f \mathbf{D}_x^{n-k} \mathbf{S}_x^k {\bf u}\big)\;.
\end{equation}
Notice that, by (\ref{33}),
\begin{equation}\label{3new}
\mathbf{S}_x\mathbf{D}_x^{n-k} \mathbf{S}_x^k {\bf u}
=\frac{1}{\alpha_{n+1-k}}\Big(\alpha\mathbf{D}_x^{n-k} \mathbf{S}_x^{k+1}{\bf u}
-\gamma_{n-k}\texttt{U}_1\mathbf{D}_x^{n+1-k} \mathbf{S}_x^{k}{\bf u}\Big)\;.
\end{equation}
Therefore, using successively (\ref{a}), (\ref{3new}), (\ref{gamma-n-bis1}), and (\ref{Tnk}), we may write
\begin{align*}
&\mathbf{D}_x\big(\mathrm{T}_{n,k}f \mathbf{D}_x^{n-k} \mathbf{S}_x^k {\bf u}\big) \\
&\quad=\big(\mathrm{S}_x\mathrm{T}_{n,k}f-\alpha^{-1}\texttt{U}_1\mathrm{D}_x \mathrm{T}_{n,k}f\big)
\mathbf{D}_x^{n+1-k} \mathbf{S}_x^k {\bf u}
+\alpha^{-1} \mathrm{D}_x \mathrm{T}_{n,k}f \mathbf{S}_x \mathbf{D}_x^{n-k} \mathbf{S}_x^k {\bf u} \\
&\quad=\Big(\mathrm{S}_x\mathrm{T}_{n,k}f-\frac{\gamma_{n+1-k}}{\alpha_{n+1-k}}\texttt{U}_1\mathrm{D}_x \mathrm{T}_{n,k}f\Big)
\mathbf{D}_x^{n+1-k} \mathbf{S}_x^k {\bf u}
+\frac{\mathrm{D}_x \mathrm{T}_{n,k}f}{\alpha_{n+1-k}} \mathbf{D}_x^{n-k} \mathbf{S}_x^{k+1} {\bf u} \\
&\quad=\Big(\mathrm{T}_{n+1,k}f-\frac{\mathrm{D}_x \mathrm{T}_{n,k-1}f}{\alpha_{n+2-k}}\Big)
\mathbf{D}_x^{n+1-k} \mathbf{S}_x^k {\bf u}
+\frac{\mathrm{D}_x \mathrm{T}_{n,k}f}{\alpha_{n+1-k}} \mathbf{D}_x^{n-k} \mathbf{S}_x^{k+1} {\bf u}\,.
\end{align*}
Substituting this expression in the right-hand side of \eqref{eq-Leib1} and then
applying the method of telescoping sums, we get
$$
\mathbf{D}^{n+1}_x \big(f{\bf u}\big)
=\sum_{k=0}^n \mathrm{T}_{n+1,k}f \mathbf{D}_x^{n+1-k} \mathbf{S}_x^k {\bf u}
+\frac{\mathrm{D}_x\mathrm{T}_{n,n}f}{\alpha_1}\mathbf{S}_x^{n+1}{\bf u}
-\frac{\mathrm{D}_x\mathrm{T}_{n,-1}f}{\alpha_{n+2}}\mathbf{D}_x^{n+1}{\bf u}\,.
$$
Finally, since $\mathrm{T}_{n,-1}f=0$ and $\frac{1}{\alpha_1}\mathbf{D}_x\mathrm{T}_{n,n}f=\mathrm{T}_{n+1,n+1}f$
(this last equality follows from \eqref{Tnk} taking therein $k=n$
and then in the resulting expression shifting $n$ into $n+1$),
we obtain (\ref{leibnizfor-NUL}) with $n$ replaced by $n+1$.
Thus (\ref{leibnizfor-NUL}) is proved.
\end{proof}

\begin{corollary}\label{LeibnizCor1}
Consider the lattice $x(s):=\mathfrak{c}_1 q^{-s} +\mathfrak{c}_2 q^s +\mathfrak{c}_3$. 
Let ${\bf u}\in\mathcal{P}^*$ and $f\in\mathcal{P}_2$. Write $f(z)=az^2+bz+c\,,$ with $a,b,c\in\mathbb{C}$. Then
\begin{align}\label{leibnizfor-degree-pi-2}
\mathbf{D}_x ^n (f{\bf u})
&=\left(\frac{a\alpha}{\alpha_n\alpha_{n-1}}\,(z-\mathfrak{c}_3)^2
+\frac{f'(\mathfrak{c}_3)}{\alpha_n}(z-\mathfrak{c}_3) +f(\mathfrak{c}_3)+\frac{4a(1-\alpha^2)\gamma_n \mathfrak{c}_1\mathfrak{c}_2}{\alpha_{n-1}}\right) \mathbf{D}_x^n{\bf u}  \\
&\quad+\frac{\gamma_n}{\alpha_n}\left( \frac{a(\alpha_n +\alpha \alpha_{n-1}) }{\alpha_{n-1}^2}\,(z-\mathfrak{c}_3)
+f'(c_3)\right) \mathbf{D}_x^{n-1}\mathbf{S}_x{\bf u}\nonumber \\
&\quad+\frac{a\gamma_n\gamma_{n-1}}{\alpha_{n-1}^2}\,\mathbf{D}_x^{n-2}\mathbf{S}_x^2{\bf u}  \nonumber
\end{align}
$(n=0,1,\ldots)$.
In particular, 
\begin{align} \label{leib-for-degree-pi-1}
\mathbf{D}_x ^n \big((bz+c){\bf u}\big)= \left(\frac{b(z-\mathfrak{c}_3)}{\alpha_n}+b\mathfrak{c}_3+c\right)\mathbf{D}_x^n{\bf u} +\frac{b\gamma_n}{\alpha_n}\ \mathbf{D}_x^{n-1} \mathbf{S}_x {\bf u}\,.
\end{align}
\end{corollary}

\begin{proof}
Since \eqref{leib-for-degree-pi-1} is the particular case  of \eqref{leibnizfor-degree-pi-2} for $a=0$,
we only need to prove \eqref{leibnizfor-degree-pi-2}.
This can be proved combining the Leibniz formula (\ref{leibnizfor-NUL}) and  identities (\ref{a}) and (\ref{33}). 
Alternatively, we may apply induction on $n$, as follows. 
Define $g(z):=f(z-\mathfrak{c}_3)=a(z-\mathfrak{c}_3)^2 +b(z-\mathfrak{c}_3)+c$. We need to show that 
\begin{align}
&(\mathrm{T}_{n,0}g)(z)=  g\left(\frac{z-\mathfrak{c}_3}{\alpha_n} +\mathfrak{c}_3  \right) + \frac{a\gamma_n}{\alpha_{n-1}} \texttt{U}_2 \left(\frac{z-\mathfrak{c}_3}{\alpha_n} +\mathfrak{c}_3  \right) ,\label{Tn0-a} \\
&(\mathrm{T}_{n,1}g)(z)= \frac{\gamma_n}{\alpha_n}\left( \frac{a(\alpha_n +\alpha \alpha_{n-1}) }{\alpha_{n-1}^2}\,(z-\mathfrak{c}_3)
+b\right),  \label{Tn1-a}\\
&(\mathrm{T}_{n,2}g)(z)=\frac{a\gamma_n\gamma_{n-1}}{\alpha_{n-1} ^2} \label{Tn2-a}
\end{align}
for each $n=0,1,2,\ldots$, where $\mathrm{T}_{n,k}f$ is defined by \eqref{T00}--\eqref{Tnk}.
Note that 
\begin{align}\label{TrueTn0-a}
(\mathrm{T}_{n,0}g)(z)=\frac{\alpha a}{\alpha_n \alpha_{n-1}}\left(z-\mathfrak{c}_3\right)^2 +\frac{b}{\alpha_n}\left(z-\mathfrak{c}_3\right) +c+\frac{4a(1-\alpha^2)\gamma_n}{\alpha_{n-1}}\mathfrak{c}_1\mathfrak{c}_2.
\end{align}
We proceed by induction on $n$.
Setting $n=0$ in \eqref{Tn0-a}--\eqref{Tn2-a}, we obtain $\mathrm{T}_{0,0}g=g$ and $\mathrm{T}_{0,1}g=0=\mathrm{T}_{0,2}g$. This agrees with \eqref{T00}--\eqref{Tnk}. Next suppose that \eqref{Tn0-a}--\eqref{Tn2-a} hold for all positive integers up to a fixed $n$. Then, by \eqref{Tnk}, we have
\begin{align}
&\mathrm{T}_{n+1,0}g=\mathrm{S}_x (\mathrm{T}_{n,0}g) -\frac{\gamma_{n+1}}{\alpha_{n+1}}\texttt{U}_1\mathrm{D}_x (\mathrm{T}_{n,0}g), \label{Htn0-a}\\
&\mathrm{T}_{n+1,1}g= \mathrm{S}_x (\mathrm{T}_{n,1}g) -\frac{\gamma_{n}}{\alpha_{n}}\texttt{U}_1\mathrm{D}_x (\mathrm{T}_{n,1}g)+\frac{1}{\alpha_{n+1}} \mathrm{D}_x (\mathrm{T}_{n,0}g) ,\label{Htn1-a}\\
&\mathrm{T}_{n+1,2}g=\mathrm{S}_x (\mathrm{T}_{n,2}g) +\frac{1}{\alpha_n} \mathrm{D}_x (\mathrm{T}_{n,1}g). \label{Htn2-a}
\end{align}
Using the identities 
\begin{align*}
\mathrm{S}_x \left((z-\mathfrak{c}_3)^2\right)&=(2\alpha^2 -1)(z-\mathfrak{c}_3)^2 +4(1-\alpha^2)\mathfrak{c}_1\mathfrak{c}_2\,, \\
\mathrm{D}_x \left((z-\mathfrak{c}_3)^2\right)&=2\alpha(z-\mathfrak{c}_3) \,, \\
\mathrm{S}_x \left((z-\mathfrak{c}_3)\right)&=\alpha(z-\mathfrak{c}_3),
\end{align*} 
we find 
\begin{align*}
\mathrm{S}_x(\mathrm{T}_{n,0}g)(z)&=\frac{\alpha(2\alpha^2 -1) a}{\alpha_n \alpha_{n-1}}\left(z-\mathfrak{c}_3\right)^2 +\frac{\alpha b}{\alpha_n}\left(z-\mathfrak{c}_3\right)+c \\
&\quad+\frac{4a(1-\alpha^2)(\alpha+\alpha_n\gamma_n)}{\alpha_n\alpha_{n-1}}\mathfrak{c}_1\mathfrak{c}_2,\\
\mathrm{S}_x(\mathrm{T}_{n,1}g)(z)&=\frac{\gamma_n}{\alpha_n}\left( \frac{\alpha a(\alpha_n +\alpha\alpha_{n-1})}{\alpha_{n_1} ^2}(z-\mathfrak{c}_3)+b \right),\\
\mathrm{D}_x(\mathrm{T}_{n,0}g)(z)&=\frac{1}{\alpha_n}\left( \frac{2\alpha^2 a}{\alpha_{n-1}}(z-\mathfrak{c}_3)+b \right),\\
\mathrm{D}_x(\mathrm{T}_{n,1}g)(z)&=\frac{a\gamma_n(\alpha_n +\alpha\alpha_{n-1})}{\alpha_n\alpha_{n-1} ^2}.
\end{align*}
Therefore, from \eqref{Htn0-a} and using \eqref{gamma-n-bis}--\eqref{form-5}, we obtain
\begin{align*}
(\mathrm{T}_{n+1,0}g)(z)&=\frac{\alpha a}{\alpha_n\alpha_{n-1}}\left(2\alpha^2 -1+2\alpha\frac{(1-\alpha^2)\gamma_{n+1}}{\alpha_{n+1}} \right)(z-\mathfrak{c}_3)^2 \\
&\quad+\frac{b(\alpha\alpha_{n+1} +(1-\alpha^2)\gamma_{n+1})}{\alpha_n\alpha_{n+1}}(z-\mathfrak{c}_3) +c\\
&\quad+\frac{4a(1-\alpha^2)(\alpha +\alpha_n\gamma_n)}{\alpha_n\alpha_{n-1}}\mathfrak{c}_1\mathfrak{c}_2 \\
&=\frac{\alpha a}{\alpha_n \alpha_{n+1}}\left(z-\mathfrak{c}_3\right)^2 +\frac{b}{\alpha_{n+1}}\left(z-\mathfrak{c}_3\right)+c+\frac{4a(1-\alpha^2)\gamma_{n+1}}{\alpha_{n}}\mathfrak{c}_1\mathfrak{c}_2.
\end{align*}
Hence \eqref{Tn0-a} holds for all $n$. Similarly, from \eqref{Htn1-a}, and using again \eqref{gamma-n-bis}--\eqref{form-5} and the identity $\alpha_{n+1}\gamma_n(\alpha_n +\alpha\alpha_{n-1})+2\alpha^2 \alpha_n=\alpha_{n-1}\gamma_{n+1}(\alpha_{n+1}+\alpha\alpha_n)$, we get
\begin{align*}
(\mathrm{T}_{n+1,1}g)(z)&= \frac{a}{\alpha_n\alpha_{n-1}}\left(\frac{\gamma_n(\alpha_n +\alpha\alpha_{n-1})}{\alpha_n} +\frac{2\alpha^2}{\alpha_{n+1}}  \right)  (z-\mathfrak{c}_3)+\frac{b}{\alpha_n}\left(\gamma_n +\frac{1}{\alpha_{n+1}}  \right)\\
&=\frac{\gamma_{n+1}}{\alpha_{n+1}}\left( \frac{a(\alpha_{n+1} +\alpha \alpha_{n}) }{\alpha_{n} ^2}\,(z-\mathfrak{c}_3)
+b\right).
\end{align*}
Hence \eqref{Tn1-a} holds for $n=0,1,\ldots$.
Finally, from \eqref{Htn2-a} it is obvious that \eqref{Tn2-a} holds with $n$ replaced by $n+1$ and, consequently, it holds for all $n$. 
Therefore \eqref{Tn0-a}--\eqref{Tn2-a} hold and so \eqref{leibnizfor-degree-pi-2} is proved.
\end{proof}

\section{Classical OPS and Rodrigues formula}
\label{SecRodrigues}

This section concerns classical OPS on lattices and their associated regular functionals. 
We start by reviewing some basic definitions and then we derive a functional version of the Rodrigues formula. 

\begin{definition}\label{NUL-def}
Let $x(s)$ be a lattice given by \eqref{xs-def}. 
${\bf u}\in\mathcal{P}^*$ is called $x-$classical if it is regular and there exist nonzero polynomials
$\phi\in\mathcal{P}_2$ and $\psi\in\mathcal{P}_1$
such that
\begin{equation}\label{NUL-Pearson}
\mathbf{D}_x(\phi{\bf u})=\mathbf{S}_x(\psi{\bf u})\;.
\end{equation}
An OPS with respect to a $x-$classical functional will be called a $x-$classical OPS
(or a classical OPS on the lattice $x$).
\end{definition}
Definition \ref{NUL-def} appears in \cite{FK-NM2011}, and extends the definition of $\mathrm{D}-$classical functional (cf. Section \ref{introduction}). 
We will refer to (\ref{NUL-Pearson}) as {\it $x-$Geronimus--Pearson functional equation on the lattice $x$,} or, simply, {\it $x-$GP functional equation}.
As mentioned in the introductory section, our principal goal in this work is to state necessary and sufficient conditions,
involving only $\phi$ and $\psi$ (or, equivalently, their coefficients),
such that a given functional ${\bf u}\in\mathcal{P}^*$ satisfying the
$x-$GP functional equation (\ref{NUL-Pearson}) becomes regular.
In order to move on we need to introduce some notation and to prove some preliminary properties.

We denote by $P_n ^{[k]}$ the monic polynomial of degree $n$ defined by
\begin{align}
P_n ^{[k]} (z):=\frac{D_x ^k P_{n+k} (z)}{ \prod_{j=1} ^k \gamma_{n+j}} =\frac{\gamma_{n} !}{\gamma_{n+k} !} D_x ^k P_{n+k} (z) \quad (k,n=0,1,\ldots). \label{Pnkx}
\end{align}
Here, as usual, it is understood that $\mathrm{D}_x ^0 f=f $, empty product equals one, and 
$$\gamma_0 !:=1\;,\quad \gamma_{n+1} !:=\gamma_1 ...\gamma_n \gamma_{n+1} \quad(n=0,1,\ldots)\,.$$
\begin{definition}
Let $\phi\in\mathcal{P}_2$ and $\psi\in\mathcal{P}_1$.
$(\phi,\psi)$ is called an $x-$admissible pair if
$$
d_n\equiv d_n(\phi,\psi,x):=\mbox{$\frac12$}\,\gamma_n\,\phi^{\prime\prime}+\alpha_n\psi'\neq0 \quad (n=0,1,\ldots).
$$
\end{definition}
This is an analogous for lattices of the corresponding definitions for the $\mathrm{D}-$classical and $(q,\omega)-$classical  cases (cf. \cite{M1991,MP1994,RKDP2020}).

\subsection{Preliminary properties}

Following \cite{FK-NM2011}, given ${\bf u}\in\mathcal{P}^*$, $\phi\in\mathcal{P}_2$, and $\psi\in\mathcal{P}_1$,
we define recursively polynomials $\phi^{[k]}\in\mathcal{P}_2$ and $\psi^{[k]}\in\mathcal{P}_1$ by
\begin{align}
& \phi^{[0]}:=\phi\;,\quad \psi^{[0]}:=\psi\;, \label{def-phi0psi0}\\
& \phi^{[k+1]}:=\mathrm{S}_x\phi^{[k]}+\texttt{U}_1\mathrm{S}_x\psi^{[k]}+\alpha\texttt{U}_2\mathrm{D}_x\psi^{[k]}\;, \label{def-phik}\\
& \psi^{[k+1]}:=\mathrm{D}_x\phi^{[k]}+\alpha\mathrm{S}_x\psi^{[k]}+\texttt{U}_1\mathrm{D}_x\psi^{[k]}\;, \label{def-psik}
\end{align}
and functionals ${\bf u}^{[k]}\in\mathcal{P}^*$ by
\begin{equation}\label{uk-func-Dx}
{\bf u}^{[0]}:={\bf u}\;,\quad
{\bf u}^{[k+1]}:=\mathbf{D}_x\big(\texttt{U}_2\psi^{[k]}{\bf u}^{[k]}\big)-\mathbf{S}_x\big(\phi^{[k]}{\bf u}^{[k]}\big)
\end{equation}
($k=0,1,\ldots$).
${\bf u}^{[k]}$ may be seen as the higher order $x-$derivative of ${\bf u}$.
Next, we provide explicit representations for the polynomials $\phi^{[k]}$ and $\psi^{[k]}$.

\begin{proposition}
Consider the lattice $x(s):=\mathfrak{c}_1 q^{-s} +\mathfrak{c}_2 q^s +\mathfrak{c}_3$.
Let $\phi\in\mathcal{P}_2$ and $\psi\in\mathcal{P}_1$, so there are $a,b,c,d,e\in\mathbb{C}$ such that
$$\phi(z)=az^2+bz+c\;,\quad \psi(z)=dz+e\,.$$ 
Then the polynomials $\phi^{[k]}$ and $\psi^{[k]}$
defined by \eqref{def-phi0psi0}--\eqref{def-psik} are given by 
\begin{align}
\psi^{[k]}(z)&=\big(a\gamma_{2k}+d\alpha_{2k}\big)(z-\mathfrak{c}_3)
+\phi'(\mathfrak{c}_3)\gamma_k+\psi(\mathfrak{c}_3)\alpha_k\;, \label{psi-explicit}\\
\phi^{[k]}(z)&=\big(d(\alpha^2-1)\gamma_{2k}+a\alpha_{2k}\big)
\big((z-\mathfrak{c}_3)^2-2\mathfrak{c}_1\mathfrak{c}_2\big)  \label{phi-explicit} \\
&\quad +\big(\phi'(\mathfrak{c}_3)\alpha_k+\psi(\mathfrak{c}_3)(\alpha^2-1)\gamma_k\big)(z-\mathfrak{c}_3)
+ \phi(\mathfrak{c}_3)+2a\mathfrak{c}_1\mathfrak{c}_2, \nonumber
\end{align}
for each $k=0,1,2\ldots$.
\end{proposition}

\begin{proof}
Set
\begin{equation}\label{phikpsik}
\phi^{[k]}(z)=a^{[k]}z^2+b^{[k]}z+c^{[k]}\;,\quad \psi^{[k]}(z)=d^{[k]}z+e^{[k]} \;,
\end{equation}
where $a^{[k]},b^{[k]},c^{[k]},d^{[k]},e^{[k]}\in\mathbb{C}$. Clearly, by \eqref{def-phi0psi0},
$$
a^{[0]}=a\;,\quad b^{[0]}=b\;,\quad c^{[0]}=c\;,\quad d^{[0]}=d\;,\quad e^{[0]}=e\;.\quad
$$
In order to determine the coefficients $a^{[k]}$, $b^{[k]}$, $c^{[k]}$, $d^{[k]}$, and $e^{[k]}$ for each $k=1,2,\ldots$, we proceed as follows. Firstly we replace in \eqref{def-phik} and in \eqref{def-psik} the expressions of $\phi^{[k]}$, $\phi^{[k+1]}$, $\psi^{[k]}$, and $\psi^{[k+1]}$ given by \eqref{phikpsik}; and then, in the two resulting identities,
using \eqref{Dx-xn} together with \eqref{U1x} and \eqref{U2x}, 
after identification of the coefficients of the polynomials appearing in both sides of each of those identities, we obtain a system with five difference equations, namely
\begin{align}
& a^{[k+1]}=(2\alpha^2-1)a^{[k]}+2\alpha(\alpha^2-1)d^{[k]}\,,  \label{EqSy1} \\
& b^{[k+1]}=\alpha b^{[k]} +(\alpha^2-1)e^{[k]}+2\beta(2\alpha+1)a^{[k]} +\beta(\alpha+1)(4\alpha-1) d^{[k]}\,, \label{EqSy3} \\
& c^{[k+1]}=c^{[k]} +\widehat{v}_2 a^{[k]}+\beta b^{[k]}+\beta(\alpha+1) e^{[k]}
+\big(\beta^2(\alpha+1)+\alpha\delta\big)d^{[k]}\,, \label{EqSy5} \\
& d^{[k+1]}=2\alpha a^{[k]} +(2\alpha^2-1)d^{[k]} \,,\label{EqSy2} \\
& e^{[k+1]}=b^{[k]} +\alpha e^{[k]}+2\beta a^{[k]}+\beta(2\alpha+1) d^{[k]} \label{EqSy4}
\end{align}
($k=0,1,\ldots$).
The explicit solution of this system is 
\begin{align}
a^{[k]} &=d(\alpha^2-1)\gamma_{2k}+a\alpha_{2k}\;,  \label{solEqSy1} \\
b^{[k]} &=\psi(\mathfrak{c}_3)(\alpha^2-1)\gamma_k+\phi'(\mathfrak{c}_3)\alpha_k
-2\mathfrak{c}_3\big(d(\alpha^2-1)\gamma_{2k}+a\alpha_{2k}\big)\;,\label{solEqSy1a} \\
c^{[k]} &=\phi(\mathfrak{c}_3)+2a\mathfrak{c}_1\mathfrak{c}_2
-\mathfrak{c}_3\big(\psi(\mathfrak{c}_3)(\alpha^2-1)\gamma_k+\phi'(\mathfrak{c}_3)\alpha_k\big) \\
&\quad+(\mathfrak{c}_3^2-2\mathfrak{c}_1\mathfrak{c}_2)\big(d(\alpha^2-1)\gamma_{2k}+a\alpha_{2k}\big)  \;, \nonumber \\
d^{[k]} &=a\gamma_{2k}+d\alpha_{2k}\;, \label{solEqSy2} \\
e^{[k]} &=\phi'(\mathfrak{c}_3)\gamma_k+\psi(\mathfrak{c}_3)\alpha_k
-\mathfrak{c}_3\big(a\gamma_{2k}+d\alpha_{2k}\big) \,.\label{solEqSy3}
\end{align}
This can be easily proved by induction on $k$. The representations (\ref{psi-explicit})--(\ref{phi-explicit}) are obtained by inserting (\ref{solEqSy1})--(\ref{solEqSy3}) into (\ref{phikpsik}).
\end{proof}

Lemma \ref{FuncEq-uk} in bellow is proved in \cite{FK-NM2011}. We point out that the statement of the result given in \cite{FK-NM2011} assumes {\it a priori} that ${\bf u}$ is a regular functional. However, inspection of the proof given therein shows that the result remains unchanged without such assumption.

\begin{lemma}\label{FuncEq-uk}
Let ${\bf u}\in\mathcal{P}^*$.
Suppose that there exist $\phi\in\mathcal{P}_2$ and $\psi\in\mathcal{P}_1$ such that
\eqref{NUL-Pearson} holds.
Then ${\bf u}^{[k]}$ satisfies the functional equation
\begin{equation}\label{uk-funcEq}
\mathbf{D}_x\big(\phi^{[k]}{\bf u}^{[k]}\big)=\mathbf{S}_x\big(\psi^{[k]}{\bf u}^{[k]}\big) \quad (k=0,1,\ldots).
\end{equation}
\end{lemma}
The next result gives some additional functional equations fulfilled by ${\bf u}^{[k]}$.
\begin{lemma}\label{lemmaA}
Let ${\bf u} \in \mathcal{P}^*$ and suppose that there exist $\phi\in\mathcal{P}_2$ and $\psi\in\mathcal{P}_1$ such that ${\bf u}$ satisfies the $x-$GP functional equation  \eqref{NUL-Pearson}. Then the relations
\begin{align}
\mathbf{D}_x\big({\bf u}^{[k+1]}\big)
&=-\alpha\psi^{[k]}{\bf u}^{[k]}\;, \label{lemmaA1} \\
\mathbf{S}_x\big({\bf u}^{[k+1]}\big)
&=-\alpha\big(\alpha\phi^{[k]}+\mbox{\rm $\texttt{U}$}_1\psi^{[k]}\big){\bf u}^{[k]}\;, \label{lemmaA2} \\
2\mbox{\rm $\texttt{U}$}_1{\bf u}^{[k+1]}
&=\mathbf{S}_x\big(\mbox{\rm $\texttt{U}$}_2\psi^{[k]}{\bf u}^{[k]}\big)
-\mathbf{D}_x\big(\mbox{\rm $\texttt{U}$}_2\phi^{[k]}{\bf u}^{[k]}\big) \label{lemmaA3}
\end{align}
hold for each $k=0,1,\ldots$.
\end{lemma}

\begin{proof}
Using \eqref{2bb} and \eqref{uk-funcEq}, we deduce 
\begin{align*}
\mathbf{D}_x^2\big(\texttt{U}_2\psi^{[k]}{\bf u}^{[k]}\big)
&=(2\alpha-\alpha^{-1}) \mathbf{S}_x^2 \big(\psi^{[k]}{\bf u}^{[k]}\big)
+\alpha^{-1}\texttt{U}_1 \mathbf{D}_x \mathbf{S}_x \big(\psi^{[k]}{\bf u}^{[k]}\big)
-\alpha \psi^{[k]}{\bf u}^{[k]} \\
&=(2\alpha-\alpha^{-1}) \mathbf{S}_x\mathbf{D}_x \big(\phi^{[k]}{\bf u}^{[k]}\big)
+\alpha^{-1}\texttt{U}_1 \mathbf{D}_x^2 \big(\phi^{[k]}{\bf u}^{[k]}\big)
-\alpha \psi^{[k]}{\bf u}^{[k]} \\
&=\mathbf{D}_x\mathbf{S}_x\big(\phi^{[k]}{\bf u}^{[k]}\big)-\alpha \psi^{[k]}{\bf u}^{[k]}\;,
\end{align*}
where the last equality follows from \eqref{33} for $n=1$
and taking into account that $\alpha_2=2\alpha^2-1$ and $\gamma_1=1$.
Therefore, by the definition of ${\bf u}^{[k+1]}$, we obtain
$$
\mathbf{D}_x {\bf u}^{[k+1]}
=\mathbf{D}_x^2\big(\texttt{U}_2\psi^{[k]}{\bf u}^{[k]}\big)
-\mathbf{D}_x\mathbf{S}_x\big(\phi^{[k]}{\bf u}^{[k]}\big)
=-\alpha \psi^{[k]}{\bf u}^{[k]}\;.
$$
This proves \eqref{lemmaA1}. Next,
by \eqref{a} and \eqref{b}, we may write
\begin{align*}
\mathbf{D}_x\big(\texttt{U}_2\phi^{[k]}{\bf u}^{[k]}\big)
&=\left(\mathrm{S}_x \texttt{U}_2 -\alpha^{-1}\texttt{U}_1\mathrm{D}_x \texttt{U}_2\right)
\mathbf{D}_x\big(\phi^{[k]}{\bf u}^{[k]}\big) +
\alpha^{-1} \big(\mathrm{D}_x \texttt{U}_2\big) \mathbf{S}_x\big(\phi^{[k]}{\bf u}^{[k]}\big)\;, \\
\mathbf{S}_x\big(\texttt{U}_2\psi^{[k]}{\bf u}^{[k]}\big)
&=\left(\mathrm{S}_x \texttt{U}_2 -\alpha^{-1}\texttt{U}_1\mathrm{D}_x \texttt{U}_2\right)
\mathbf{S}_x\big(\psi^{[k]}{\bf u}^{[k]}\big) + \alpha^{-1} \big(\mathrm{D}_x \texttt{U}_2\big)
\mathbf{D}_x\big(\texttt{U}_2 \psi^{[k]}{\bf u}^{[k]}\big) \,.
\end{align*}
After subtracting these two equalities and taking into account \eqref{uk-funcEq}, as well as
the relation $\alpha^{-1} \mathrm{D}_x \texttt{U}_2=2\texttt{U}_1$ (cf. \eqref{DxSxU2}), we obtain \eqref{lemmaA3}.
To prove \eqref{lemmaA2}, note first that, by the definition of ${\bf u}^{[k+1]}$,
\begin{equation}\label{EqLemmaAeq1}
\alpha_2 \mathbf{S}_x {\bf u}^{[k+1]}
=\alpha_2\mathbf{S}_x\mathbf{D}_x\big(\texttt{U}_2\psi^{[k]}{\bf u}^{[k]}\big)
-\alpha_2\mathbf{S}_x^2\big(\phi^{[k]}{\bf u}^{[k]}\big) \;.
\end{equation}
Using again \eqref{33} for $n=1$, we have 
$$
\alpha_2\mathbf{S}_x\mathbf{D}_x\big(\texttt{U}_2\psi^{[k]}{\bf u}^{[k]}\big)
=\alpha\mathbf{D}_x\mathbf{S}_x\big(\texttt{U}_2\psi^{[k]}{\bf u}^{[k]}\big)
-\texttt{U}_1\mathbf{D}_x^2\big(\texttt{U}_2\psi^{[k]}{\bf u}^{[k]}\big)
$$
and by \eqref{2bb}, we also have
$$
\alpha_2\mathbf{S}_x^2\big(\phi^{[k]}{\bf u}^{[k]}\big)
=-\texttt{U}_1\mathbf{D}_x\mathbf{S}_x\big(\phi^{[k]}{\bf u}^{[k]}\big)
+\alpha^2\phi^{[k]}{\bf u}^{[k]}
+\alpha\mathbf{D}_x^2\big(\texttt{U}_2\phi^{[k]}{\bf u}^{[k]}\big)\;.
$$
Substituting these two expressions into the right-hand side of \eqref{EqLemmaAeq1}, we get
\begin{align}\label{eqq1AA}
&(2\alpha^2-1)\mathbf{S}_x {\bf u}^{[k+1]} \\
&\quad=\alpha\mathbf{D}_x\mathbf{S}_x\big(\texttt{U}_2\psi^{[k]}{\bf u}^{[k]}\big)
-\alpha\mathbf{D}_x^2\big(\texttt{U}_2\phi^{[k]}{\bf u}^{[k]}\big)
-\texttt{U}_1\mathbf{D}_x{\bf u}^{[k+1]} -\alpha^2\phi^{[k]}{\bf u}^{[k]}\;. \nonumber
\end{align}
Next, by taking $f=\texttt{U}_1$ and replacing ${\bf u}$ by ${\bf u}^{[k+1]}$ in \eqref{a},
and then using \eqref{DxSxU1} and \eqref{lemmaA1}, we derive
$$
\mathbf{D}_x\big(\texttt{U}_1{\bf u}^{[k+1]}\big)
=\texttt{U}_1 \psi^{[k]}{\bf u}^{[k]}+(\alpha-\alpha^{-1})\mathbf{S}_x {\bf u}^{[k+1]}\;.
$$
Multiplying both sides of this equality by $2\alpha$ and
combining the resulting equality with the one obtained by applying $\mathbf{D}_x$ to both sides of \eqref{lemmaA3}, we deduce
\begin{align}
(2\alpha^2-2)\mathbf{S}_x {\bf u}^{[k+1]}
=\alpha\mathbf{D}_x\mathbf{S}_x\big(\texttt{U}_2\psi^{[k]}{\bf u}^{[k]}\big)
-\alpha\mathbf{D}_x^2\big(\texttt{U}_2\phi^{[k]}{\bf u}^{[k]}\big)
+2\alpha\texttt{U}_1\psi^{[k]}{\bf u}^{[k]} \;. \label{eqq1AB}
\end{align}
Finally, subtracting \eqref{eqq1AB} to \eqref{eqq1AA},
and taking into account \eqref{lemmaA1}, \eqref{lemmaA2} follows.
\end{proof}

\subsection{A functional Rodrigues formula}

Here we prove a functional version of the Rodrigues formula on lattices, extending results stated in \cite{MP1994,RKDP2020}.

\begin{theorem}[Rodrigues' formula]\label{uk-is-classical}
Consider the lattice
$$x(s):=\mathfrak{c}_1 q^{-s} +\mathfrak{c}_2 q^s +\mathfrak{c}_3\;.$$ 
Let ${\bf u}\in\mathcal{P}^*$ and suppose that there exists a $x-$admissible pair $(\phi,\psi)$ such that
${\bf u}$ fulfills the $x-GP$ functional equation \eqref{NUL-Pearson}.
Set
\begin{equation}\label{dnen-Prop}
d_n:=\mbox{$\frac{1}{2}\,$}\phi''\gamma_n+\psi'\alpha_n\;,\quad
e_n:=\phi'(\mathfrak{c}_3)\gamma_n+\psi(\mathfrak{c}_3)\alpha_n \quad(n=0,1,\ldots).
\end{equation}
Then 
\begin{align}\label{roformula}
R_n {\bf u} = \mathbf{D}_x^n {\bf u}^{[n]} \quad (n=0,1,\ldots)\,,
\end{align}
where ${\bf u}^{[n]}$ is the functional on $\mathcal{P}$ defined by \eqref{uk-func-Dx} and $(R_n)_{n\geq0}$ is a simple set of polynomials given by the TTRR
\begin{align}\label{Rn-Prop1} 
R_{n+1}(z)=(a_n z-s_n)R_{n}(z)- t_n R_{n-1} (z) \quad(n=0,1,\ldots)\;, 
\end{align}
with initial conditions $R_{-1}=0$ and $R_0=1$,
and $(a_n)_{n\geq0}$, $(s_n)_{n\geq0}$, and $(t_n)_{n\geq1}$
are sequences of complex numbers defined by
\begin{align}
&a_n:=-\frac{\alpha\,d_{2n}d_{2n-1}}{d_{n-1}}\; , \label{rn-Prop1}\\
&s_n:=a_n \left( \mathfrak{c}_3+\frac{\gamma_n e_{n-1}}{d_{2n-2}}
-\frac{\gamma_{n+1} e_{n}}{d_{2n}}\right)\;  , \label{sn-Prop1}\\
&t_n:=a_n\,\frac{\alpha\,\gamma_n d_{2n-2}}{d_{2n-1}}\phi^{[n-1]}\left(\mathfrak{c}_3 -\frac{e_{n-1}}{d_{2n-2}}\right)\;,\label{tn-Prop1}
\end{align}
$\phi^{[n-1]}$ being given by \eqref{phi-explicit}.
(It is understood that $a_0:=-\alpha d$ and $s_0:=\alpha e$.)
\end{theorem}

\begin{proof}
We apply mathematical induction on $n$.
If $n=0$, (\ref{roformula}) is trivial.
If $n=1$, (\ref{roformula}) follows from \eqref{lemmaA1}, since $R_1=-\alpha\psi$.
Assume now (induction hypothesis) that (\ref{roformula}) holds for two consecutive nonnegative integer numbers,
i.e., the relations
\begin{align}\label{InducHyp}
R_{n-1} {\bf u}=\mathbf{D}_x^{n-1} {\bf u}^{[n-1]} \;,\quad
R_n {\bf u}=\mathbf{D}_x^n {\bf u}^{[n]}
\end{align}
hold for some fixed $n\in\mathbb{N}$.
We need to prove that $R_{n+1}{\bf u}=\mathbf{D}_x^{n+1} {\bf u}^{[n+1]}$.
Notice first that, by \eqref{psi-explicit} and \eqref{dnen-Prop}, we have
\begin{equation}\label{psid2n}
\psi^{[k]}(z)=d_{2k}(z-\mathfrak{c}_3)+e_k\quad(k=0,1,\ldots)\,.
\end{equation}
By \eqref{lemmaA1} and the Leibniz formula in Proposition \ref{Leibniz-rule-NUL}, we may write
\begin{align*}
\mathbf{D}_x^{n+1}{\bf u}^{[n+1]}
&= \mathbf{D}_x^n \mathbf{D}_x {\bf u}^{[n+1]} = -\alpha \mathbf{D}_x^n (\psi^{[n]} {\bf u}^{[n]}) \\
&= -\alpha \mathrm{T}_{n,0} \psi ^{[n]}  \mathbf{D}_x^n {\bf u}^{[n]}
-\alpha \mathrm{T}_{n,1}\psi^{[n]}\mathbf{D}_x^{n-1} \mathbf{S}_x {\bf u}^{[n]}\,.
\end{align*}
From \eqref{leib-for-degree-pi-1} we have $\mathrm{T}_{n,1}\psi^{[n]}=d_{2n}\gamma_n /\alpha_n$,
and so, using also \eqref{InducHyp},
\begin{align}
\mathbf{D}_x^{n-1} \mathbf{S}_x {\bf u}^{[n]} =-\frac{\alpha_n}{\alpha d_{2n} \gamma_n}
\left(\mathbf{D}_x ^{n+1} {\bf u}^{[n+1]}
+ \alpha \big(\mathrm{T}_{n,0}\psi ^{[n]}\big) R_n {\bf u} \right)\,. \label{1a-lattice}
\end{align}
Shifting $n$ into $n-1$, and using again the induction hypothesis \eqref{InducHyp}, we obtain
\begin{align}
\mathbf{D}_x^{n-2} \mathbf{S}_x {\bf u}^{[n-1]} =-\frac{\alpha_{n-1}}{\alpha d_{2n-2}\gamma_{n-1}}
\left(R_n+\alpha \big(\mathrm{T}_{n-1,0}\psi^{[n-1]}\big)\,R_{n-1}\right){\bf u}\;. \label{1b-lattice}
\end{align}
Next, using \eqref{lemmaA1}, \eqref{a}, and \eqref{lemmaA2}, we deduce
\begin{align}
\mathbf{D}_x^{n+1}{\bf u}^{[n+1]} &= -\alpha \mathbf{D}_x^n \big(\psi^{[n]}{\bf u}^{[n]}\big)
=-\alpha \mathbf{D}_x^{n-1}\big( \mathbf{D}_x(\psi^{[n]} {\bf u}^{[n]})\big) \label{Eq-xi1}\\
&=-\mathbf{D}_x^{n-1} \Big(\big(\alpha \mathrm{S}_x \psi ^{[n]}-\texttt{U}_1\mathrm{D}_x\psi^{[n]}\big)
\mathbf{D}_x{\bf u}^{[n]} +\mathrm{D}_x\psi ^{[n]}\mathbf{S}_x {\bf u}^{[n]} \Big) \nonumber \\
&=\mathbf{D}_x^{n-1} \big( \xi_2(\cdot;n) {\bf u}^{[n-1]}  \big)\;, \nonumber
\end{align}
where $\xi_2(\cdot;n)$ is a polynomial of degree $2$, given by
\begin{align}\label{xi-definition}
\xi_2 (z;n)= \alpha^2 \big(\psi ^{[n-1]}\mathrm{S}_x \psi ^{[n]} +\phi ^{[n-1]}\mathrm{D}_x \psi ^{[n]}   \big)(z)\;.
\end{align}
The following identities may be proved by a straightforward computation:
\begin{align*}
d_{2n-1}-\alpha d_{2n-2}&=a^{[n-1]} \;,\\
d_{2n-2}\big( e_n -2\alpha \mathfrak{c}_3d_{2n} \big) +d_{2n}\big(b^{[n-1]}+\alpha e_{n-1}\big)&=2d_{2n-1}(\alpha e_n-\mathfrak{c}_3 d_{2n})
\end{align*}
for each $n=1,2,\ldots$. (The second one is achieved by using equation \eqref{EqSy4}.)
Using these relations, together with \eqref{phikpsik}, \eqref{psid2n}, \eqref{Dx-xn}, and \eqref{Sx-xn},
we deduce
\begin{align}\label{xi-2zn}
\xi_2 (z;n) &= \alpha ^2 d_{2n}d_{2n-1} z^2 +2\alpha ^2 d_{2n-1} (\alpha e_n-\mathfrak{c}_3 d_{2n})z \\
&\quad+\alpha^2 \big(d_{2n} c^{[n-1]} +(e_{n-1}-\mathfrak{c}_3d_{2n-2} )(e_n-\alpha\mathfrak{c}_3 d_{2n})\big)\,.\nonumber
\end{align}
Since $\deg\xi_2(\cdot;n)=2$, using again Proposition \ref{Leibniz-rule-NUL}, we may write
\begin{align}
\mathbf{D}_x^{n-1}\big(\xi_2(\cdot;n){\bf u}^{[n-1]}\big)
&= \mathrm{T}_{n-1,0}\xi_2(\cdot;n) \mathbf{D}_x^{n-1}{\bf u}^{[n-1]}
+\mathrm{T}_{n-1,1}\xi_2(\cdot;n)\mathbf{D}_x^{n-2}\mathrm{S}_x{\bf u}^{[n-1]} \label{Eqq-ss} \\
&\quad+\mathrm{T}_{n-1,2}\xi_2(\cdot;n) \mathbf{D}_x^{n-3}\mathbf{S}_x^2{\bf u}^{[n-1]}\;. \nonumber
\end{align}
Since, by \eqref{leib-for-degree-pi-1},
$\mathrm{T}_{n-1,2}\xi_2(\cdot;n)=\alpha^2 \gamma_{n-1} \gamma_{n-2}d_{2n} d_{2n-1}/\alpha_{n-2}^2$,
combining equations \eqref{Eqq-ss}, \eqref{Eq-xi1}, \eqref{1b-lattice}, and \eqref{InducHyp}, we obtain
\begin{align}
\mathbf{D}_x^{n-3} \mathbf{S}_x^2{\bf u}^{[n-1]}
&=\frac{\alpha_{n-2}^2}{\alpha^2\gamma_{n-1}\gamma_{n-2}d_{2n}d_{2n-1}}
\left( \mathbf{D}_x^{n+1}{\bf u}^{[n+1]}-\big(\mathrm{T}_{n-1,0}\xi_2(\cdot;n)\big)R_{n-1}{\bf u} \right.\label{1c-lattice} \\
&\quad + \left. \frac{\alpha_{n-1} \mathrm{T}_{n-1,1}\xi_2(\cdot;n)}{\alpha\gamma_{n-1} d_{2n-2}} \Big(R_n  + \alpha \big(\mathrm{T}_{n-1,0}\psi^{[n-1]}\big)R_{n-1} \Big){\bf u} \right)\;. \nonumber
\end{align}
On the other hand, by (\ref{lemmaA2}),
\begin{align}\label{eta-definition}
\mathbf{S}_x {\bf u}^{[n]}=\eta_2(\cdot;n){\bf u}^{[n-1]}\;,\quad
\eta_2(z;n):=-\alpha\big(\alpha\phi^{[n-1]}+\texttt{U}_1\psi^{[n-1]}\big)(z)\;.
\end{align}
Taking into account that $\eta_2(\cdot;n)$ is a polynomial of degree at most two, by the Leibniz formula and \eqref{InducHyp}, we may write
\begin{align}
\mathbf{D}_x^{n-1} \mathbf{S}_x {\bf u}^{[n]}
&=\mathbf{D}_x^{n-1} \big( \eta_2 (\cdot;n){\bf u}^{[n-1]} \big) \label{EqBoa}\\
&= \big(\mathrm{T}_{n-1,0}\eta_2(\cdot;n)\big)R_{n-1}{\bf u}
+\mathrm{T}_{n-1,1}\eta_2(\cdot;n)\mathbf{D}_x^{n-2} \mathbf{S}_x{\bf u}^{[n-1]} \nonumber \\
&\quad +\mathrm{T}_{n-1,2}\eta_2(\cdot;n)\mathbf{D}_x^{n-3} \mathbf{S}_x^2{\bf u}^{[n-1]}\;. \nonumber
\end{align}
Note that $\eta_2(\cdot;n)$ is given explicitly by
\begin{align}\label{eta-2zn}
\eta_2 (z;n)&= \alpha (\alpha d_{2n-1}-d_{2n})z^2 -\alpha\Big(\alpha b^{[n-1]} +(\alpha^2-1)(e_{n-1} -2\mathfrak{c}_3 d_{2n-2})\Big)z \\
&\quad -\alpha\big(\alpha c^{[n-1]}+\beta(\alpha +1)(e_{n-1}-\mathfrak{c}_3d_{2n-2})\big)\,.\nonumber
\end{align}
Hence, using \eqref{leib-for-degree-pi-1},
$\mathrm{T}_{n-1,2}\eta_2(\cdot;n)=\alpha \gamma_{n-1}\gamma_{n-2}(\alpha d_{2n-1}-d_{2n})/\alpha_{n-2}^2$.
Therefore, substituting \eqref{1a-lattice}, \eqref{1b-lattice}, and \eqref{1c-lattice} in \eqref{EqBoa}, we obtain
\begin{equation}\label{Dn1HypInd}
\mathbf{D}_x^{n+1} {\bf u}^{[n+1]}
=\big(A(\cdot;n) R_n+ B(\cdot;n)R_{n-1}\big){\bf u}\;,
\end{equation}
where $A(\cdot;n)$ and $B(\cdot;n)$ are polynomials depending on $n$, given by
\begin{align}\label{Azn-initial}
\epsilon_n A(z;n) &= \frac{\alpha_n \big(\mathrm{T}_{n,0}\psi ^{[n]}\big)(z)}{\gamma_n d_{2n}} -\frac{\alpha_{n-1} \big(\mathrm{T}_{n-1,1} \eta_2 \big)(z;n)}{\alpha \gamma_{n-1} d_{2n-2}} \\
&\quad+\frac{\alpha_{n-1}(\alpha d_{2n-1}-d_{2n}) \big(\mathrm{T}_{n-1,1}\xi_2\big)(z;n)}{\alpha ^2 \gamma_{n-1} d_{2n}d_{2n-1} d_{2n-2}} \nonumber
\end{align}
and
\begin{align}\label{Bzn-initial}
\epsilon_n B(z;n) &= \big(\mathrm{T}_{n-1,0}\eta_2\big)(z;n) -\frac{\alpha_{n-1} \big(\mathrm{T}_{n-1,0}\psi ^{[n-1]}\big)(z)\big(\mathrm{T}_{n-1,1}\eta_2\big)(z;n)}{ \gamma_{n-1} d_{2n-2}} \nonumber \\
&\quad + \frac{(d_{2n}-\alpha d_{2n-1}) \big(\mathrm{T}_{n-1,0}\xi_2\big)(z;n)}{\alpha d_{2n} d_{2n-1}} \\
&\quad + \frac{\alpha_{n-1} (\alpha d_{2n-1}-d_{2n}) \big(\mathrm{T}_{n-1,1}\xi_2\big)(z;n)\big(\mathrm{T}_{n-1,0}\psi ^{[n-1]}\big)(z)}{\alpha \gamma_{n-1} d_{2n} d_{2n-1} d_{2n-2}}\,,\nonumber
\end{align}
where
\begin{align*}
\epsilon_n :=\frac{d_{2n} -\alpha d_{2n-1}}{\alpha d_{2n}d_{2n-1}} -\frac{\alpha_n}{\alpha \gamma_n d_{2n}}=-\frac{d_{n-1}}{\alpha \gamma_n d_{2n} d_{2n-1}}\,.
\end{align*}
Note that $\gamma_nd_{2n} -(\alpha +\alpha_n)d_{2n-1}=-d_{n-1}$ ($n=0,1,\ldots$).
We claim that
\begin{align}
A(z;n)&=-\alpha \frac{d_{2n}d_{2n-1}}{d_{n-1}}(z-\mathfrak{c}_3) +\frac{\alpha \gamma_n d_{2n}d_{2n-1}e_{n-1}}{d_{2n-2}d_{n-1}} 
-\frac{\alpha \gamma_{n+1} d_{2n-1}e_{n}}{d_{n-1}}  \label{Azn-final} \\
&=a_n z-s_n\,,\nonumber \\
B(z;n)&= \alpha ^2 \frac{\gamma_n d_{2n}d_{2n-2}}{d_{n-1}} \phi ^{[n-1]} \left(\mathfrak{c}_3 -\frac{e_{n-1}}{d_{2n-2}} \right) =-t_n \label{Bzn-final}
\end{align}
($n=0,1,\ldots$), where $a_n$, $s_n$, and $t_n$ are given by \eqref{rn-Prop1}--\eqref{tn-Prop1}. 
Indeed, by \eqref{solEqSy1} and \eqref{solEqSy1a}, 
\begin{align}
&a^{[n-1]}= d_{2n-1}-\alpha d_{2n-2}, \label{an-app}\\
&b^{[n-1]}=e_n-\alpha e_{n-1}-2\mathfrak{c}_3a^{[n-1]} ,\label{bn-app}\\
&d_{2n}-2\alpha d_{2n-1}+d_{2n-2}=0\label{dn-app}
\end{align}
($n=1,2,\ldots$).
By \eqref{xi-2zn}, \eqref{eta-2zn}, and \eqref{psid2n}, and applying \eqref{leibnizfor-degree-pi-2}--\eqref{leib-for-degree-pi-1} together with \eqref{bn-app}, we obtain
\begin{align}
(\mathrm{T}_{n,0}\psi ^{[n]})(z) &=\frac{d_{2n}}{\alpha_n}(z-\mathfrak{c}_3) +e_n \label{Tn-psi-n,0}   ,\\
(\mathrm{T}_{n-1,1} \xi_2 )(z;n)&=\frac{\gamma_{n-1}}{\alpha_{n-1} }\Big( \frac{\alpha^2 d_{2n}d_{2n-1}(\alpha_{n-1}+\alpha\alpha_{n-2})}{\alpha_{n-2} ^2}(z-\mathfrak{c}_3) \label{Tn-xi-n-1,1} \\ &\quad+2\alpha^3 e_n d_{2n-1} \Big)  , \nonumber \\
(\mathrm{T}_{n-1,1}\eta_2 )(z;n)&=\frac{\gamma_{n-1}}{\alpha_{n-1} }\Big( \frac{\alpha (\alpha d_{2n-1}-d_{2n})(\alpha_{n-1}+\alpha\alpha_{n-2})}{\alpha_{n-2} ^2}(z-\mathfrak{c}_3)   \label{Tn-eta-n-1,1}  \\ &\quad +\alpha (e_{n-1}-\alpha e_n) \Big).  \nonumber
\end{align}
Similarly,
\begin{align}
(\mathrm{T}_{n-1,1}\eta_2)(z;n)=&\frac{\alpha^2(\alpha d_{2n-1}-d_{2n})}{\alpha_{n-1} \alpha_{n-2}}(z-\mathfrak{c}_3)^2 +\frac{\alpha(e_{n-1}-\alpha e_n)}{\alpha_{n-1}}(z-\mathfrak{c}_3) \label{Tn-eta-n-1,0} \\
& +\eta_2 (\mathfrak{c}_3;n)+\frac{4\alpha (1-\alpha^2)\gamma_{n-1}(\alpha d_{2n-1}-d_{2n}) }{\alpha_{n-2}}\mathfrak{c}_1\mathfrak{c}_2\,, \nonumber \\
(\mathrm{T}_{n-1,1}\xi_2)(z;n)=&\frac{\alpha^3 d_{2n-1}d_{2n}}{\alpha_{n-1}\alpha_{n-2}}(z-\mathfrak{c}_3)^2+\frac{2\alpha^3 e_nd_{2n-1}}{\alpha_{n-1}}(z-\mathfrak{c}_3)+\xi_2 (\mathfrak{c}_3;n) \label{Tn-xi-n-1,0} \\
& +\frac{4\alpha^2 (1-\alpha^2)\gamma_{n-1} d_{2n-1}d_{2n} }{\alpha_{n-2}}\mathfrak{c}_1\mathfrak{c}_2\,.  \nonumber     
\end{align}
Using \eqref{Tn-psi-n,0}, \eqref{Tn-xi-n-1,1}, and \eqref{Tn-eta-n-1,1} together with the identity $\alpha_n +\alpha\gamma_n=\gamma_{n+1}$ 
(this last one follows from \eqref{gamma-n-bis} and \eqref{gamma-n-bis1}), we deduce from \eqref{Azn-initial} that
\begin{align*}
\epsilon_n A(z;n)&=\frac{1}{\gamma_n}(z-\mathfrak{c}_3) -\frac{e_{n-1}}{d_{2n-2}}+ \frac{\alpha_n e_n}{\gamma_n d_{2n}} +\alpha e_n\frac{2\alpha d_{2n-1}-d_{2n} }{d_{2n}d_{2n-2}}\\
&=\frac{1}{\gamma_n}(z-\mathfrak{c}_3) -\frac{e_{n-1}}{d_{2n-2}}+ \frac{\gamma_{n+1} e_n}{\gamma_n d_{2n}}
\end{align*}
($n=0,1,\ldots$). This gives \eqref{Azn-final}.
From \eqref{Tn-psi-n,0}--\eqref{Tn-xi-n-1,0} it is straightforward to verify that \eqref{Bzn-initial} reduces to 
\begin{align*}
\epsilon_n B(z;n)&=\eta_2(\mathfrak{c}_3;n)+\frac{(d_{2n}-\alpha d_{2n-1})\xi_2(\mathfrak{c}_3;n)}{\alpha d_{2n}d_{2n-1}}+\frac{\alpha(\alpha e_n -e_{n-1})e_{n-1}}{d_{2n-2}}\\
&\quad +\frac{2\alpha^2 (\alpha d_{2n-1}-d_{2n})e_ne_{n-1}}{d_{2n}d_{2n-2}}.
\end{align*} 
Moreover, from the definitions of $\xi(.;n)$ and $\eta_2(.;n)$ given in \eqref{xi-definition} and \eqref{eta-definition}, we deduce
$$ \xi(\mathfrak{c}_3;n)=\alpha^2(e_ne_{n-1}+\phi^{[n-1]}(\mathfrak{c}_3)d_{2n}),\quad \eta_2(\mathfrak{c}_3;n)=-\alpha^2\phi^{[n-1]}(\mathfrak{c}_3)$$
($n=0,1,\ldots$). Consequently, by \eqref{an-app}, we show that
$$
\epsilon_nB(z;n)=-\alpha\frac{d_{2n-2}}{d_{2n-1}}\phi^{[n-1]}(\mathfrak{c}_3)+\alpha e_{n-1}\left(\frac{e_n}{d_{2n-1}}-\frac{e_{n-1}}{d_{2n-2}}  \right)\,.
$$
Therefore, using successively \eqref{an-app} and \eqref{bn-app}, we obtain
\begin{align*}
B(z;n)&= \frac{\alpha^2 \gamma_n d_{2n}d_{2n-2}}{d_{n-1}}\left( d_{2n-1}\frac{e_{n-1} ^2}{d_{2n-2} ^2} -e_n\frac{e_{n-1}}{d_{2n-2}} +\phi^{[n-1]}(\mathfrak{c}_3)  \right) \\
&= \frac{\alpha^2 \gamma_n d_{2n}d_{2n-2}}{d_{n-1}}\left((a^{[n-1]}+\alpha d_{2n-2})\frac{e_{n-1} ^2}{d_{2n-2} ^2} \right. \\
&\quad \left. -(b^{[n-1]}+\alpha e_{n-1}+2\mathfrak{c}_3 a^{[n-1]})\frac{e_{n-1}}{d_{2n-2}} +\phi^{[n-1]}(\mathfrak{c}_3)  \right) \\
&=\frac{\alpha^2 \gamma_n d_{2n}d_{2n-2}}{d_{n-1}}\phi^{[n-1]}\left(\mathfrak{c}_3-\frac{e_{n-1}}{d_{2n-2}} \right),
\end{align*}
and \eqref{Bzn-final} is proved.
It follows from \eqref{Azn-final} and \eqref{Bzn-final} that  \eqref{Dn1HypInd} reduces to $\mathbf{D}_x^{n+1} {\bf u}^{[n+1]}=R_{n+1}{\bf u}$,
which completes the proof.
\end{proof}

\subsection{Regularity of ${\bf u}^{[k]}$}

The next result is virtually proved in \cite[Theorem 2]{MP1994}.

\begin{lemma}\label{phi-psi-not-zero}
Let ${\bf u}\in\mathcal{P}^*$ be regular. Suppose that there is 
$(\phi,\psi)\in\mathcal{P}_2\times\mathcal{P}_1 \setminus \{(0,0)\}$ so that \eqref{NUL-Pearson} holds. 
Then neither $\phi$ nor $\psi$ is the zero polynomial, and $\deg\psi=1$.
\end{lemma}

Lemma \ref{x-admissible} in bellow gives the regularity of ${\bf u}^{[k]}$. The result appears in \cite[Proposition 4]{FK-NM2011}. 
However the proof of the $x-$admissibility condition given therein is incorrect. 
We present a proof following ideas presented in \cite{MP1994,RKDP2020}.

\begin{lemma}\label{x-admissible}
Let ${\bf u}\in\mathcal{P}^*$. Suppose that ${\bf u}$ is regular and satisfies $(\ref{NUL-Pearson})$, 
where $(\phi,\psi)\in\mathcal{P}_2\times\mathcal{P}_1 \setminus \{(0,0)\}$.
Then $(\phi,\psi)$ is a $x-$admissible pair and ${\bf u}^{[k]}$ is regular for each $k=1,2,\ldots$.
Moreover, if $(P_n)_{n\geq0}$ is the monic OPS with respect to ${\bf u}$, then
$\big(P_n^{[k]}\big)_{n\geq0}$ is the monic OPS with respect to ${\bf u}^{[k]}$.
\end{lemma}

\begin{proof}
Set $\phi(z)=az^2+bz+c$ and $\psi(z)=dz+e$. 
By Lemma \ref{phi-psi-not-zero}, both $\phi$ and $\psi$ are nonzero polynomials, and $d\neq0$.
If $\deg\phi\in\{0,1\}$ then $d_n=d\alpha_n\neq0$ for each $n=0,1,\ldots$, and so the pair $(\phi,\psi)$ is $x-$admissible. 
Assume now that $\deg\phi=2$. Then $d_n=a\gamma_n+d\alpha_n$, with $a\neq0$. To prove that $d_n\neq0$, we start by showing that
\begin{equation}\label{Equa1}
\big\langle{\bf u},\big(\texttt{U}_2\psi\mathrm{D}_xP_n^{[1]}+\phi\mathrm{S}_xP_n^{[1]}\big)P_{n+2}\big\rangle
=-\big\langle{\bf u}^{[1]},\big(\mathrm{S}_xP_{n+2}+\alpha^{-1}\texttt{U}_1\mathrm{D}_xP_{n+2}\big)P_n^{[1]}\big\rangle
\end{equation}
for each $n=0,1,\ldots$. 
Indeed, we have
\begin{align*}
&\big\langle{\bf u},\big(\texttt{U}_2\psi\mathrm{D}_xP_n^{[1]}+\phi\mathrm{S}_xP_n^{[1]}\big)P_{n+2}\big\rangle \\
&\qquad=\big\langle \texttt{U}_2\psi{\bf u},P_{n+2}\mathrm{D}_xP_n^{[1]}\big\rangle
+\big\langle \phi{\bf u},P_{n+2}\mathrm{S}_xP_n^{[1]}\big\rangle \\
&\qquad=\big\langle \texttt{U}_2\psi{\bf u},\mathrm{D}_x\big((\mathrm{S}_xP_{n+2} -\alpha^{-1}\texttt{U}_1\mathrm{D}_xP_{n+2})P_n^{[1]}\big)
-\alpha^{-1}\mathrm{S}_x\big(P_n^{[1]}\mathrm{D}_xP_{n+2}\big)\big\rangle \\
&\qquad\quad+\big\langle \phi{\bf u},\mathrm{S}_x\big((\mathrm{S}_xP_{n+2} -\alpha^{-1}\texttt{U}_1\mathrm{D}_xP_{n+2})P_n^{[1]}\big) -\alpha^{-1}\texttt{U}_2\mathrm{D}_x\big(P_n^{[1]}\mathrm{D}_xP_{n+2}\big)\big\rangle \\
&\qquad=-\big\langle {\bf u}^{[1]},\big(\mathrm{S}_xP_{n+2}-\alpha^{-1}\texttt{U}_1\mathrm{D}_xP_{n+2}\big)P_n^{[1]}\big\rangle \\
&\qquad\quad-\alpha^{-1}\big\langle \mathbf{S}_x(\texttt{U}_2\psi{\bf u})-\mathbf{D}_x(\texttt{U}_2\phi{\bf u}), P_n^{[1]}\mathrm{D}_xP_{n+2}\big\rangle\;,
\end{align*}
where the second equality holds by (\ref{PropNUL1}) and (\ref{PropNUL2}).
Therefore, using \eqref{lemmaA3} for $n=0$, we obtain (\ref{Equa1}).
Now, on the one hand, $\texttt{U}_2\psi\mathrm{D}_xP_n^{[1]}+\phi\mathrm{S}_xP_n^{[1]}$
is a polynomial of degree at most $n+2$,
being the coefficient of $z^{n+2}$ equal to $(\alpha^2-1)d\gamma_n+a\alpha_n$.
Hence, since the relations
$$(\alpha^2-1)d\gamma_n+a\alpha_n=d_{n+1}-\alpha d_n=\alpha d_n-d_{n-1}\quad(n=1,2,\ldots)$$
hold, we get
$$
\texttt{U}_2\psi\mathrm{D}_xP_n^{[1]}+\phi\mathrm{S}_xP_n^{[1]}=\big(\alpha d_n-d_{n-1}\big)z^{n+2}+(\mbox{\rm lower degree terms})
$$
for each $n=1,2,\ldots$. Consequently,
\begin{equation}\label{Equa2}
\big\langle{\bf u},\big(\texttt{U}_2\psi\mathrm{D}_xP_n^{[1]}+\phi\mathrm{S}_xP_n^{[1]}\big)P_{n+2}\big\rangle
=(\alpha d_n-d_{n-1})\langle{\bf u},P_{n+2}^2\rangle\quad(n=1,2,\ldots)\,.
\end{equation}
On the other hand, since
$\mathrm{S}_xP_{n+2}+\alpha^{-1}\texttt{U}_1\mathrm{D}_xP_{n+2}=\sum_{j=0}^{n+2}c_{n,j}P_j^{[1]}$
for some coefficients $c_{n,0},\ldots,c_{n,n+2}\in\mathbb{C}$, and using the next equation for $k=1$ 
\begin{align}\label{uk-is-regular}
\big\langle{\bf u}^{[k]},P_n^{[k]}P_m^{[k]}\big\rangle
=\alpha \frac{ d_n ^{[k-1]}}{\gamma_{n+1}}\langle{\bf u},(P_{n+1} ^{[k-1]})^2\rangle\delta_{n,m}\quad(0\leq m\leq n\,;\;n=0,1,\ldots)
\end{align}
(see \cite[Proof of Theorem 5 -- step 1.1]{FK-NM2011}), we obtain
\begin{equation}\label{Equa3}
\big\langle{\bf u}^{[1]},\big(\mathrm{S}_xP_{n+2}+\alpha^{-1}\texttt{U}_1\mathrm{D}_xP_{n+2}\big)P_n^{[1]}\big\rangle
=\frac{\alpha c_{n,n} d_n}{\gamma_{n+1}}\langle{\bf u},P_{n+1}^2\rangle\quad(n=1,2,\ldots)\,.
\end{equation}
Substituting (\ref{Equa2}) and (\ref{Equa3}) into (\ref{Equa1}), and since
$C_{n+2}=\langle{\bf u},P_{n+2}^2\rangle/\langle{\bf u},P_{n+1}^2\rangle$, we deduce
$$
\alpha\left(1+\frac{c_{n,n}}{\gamma_{n+1}C_{n+2}} \right) d_n=d_{n-1}\quad(n=1,2,\ldots)\;.
$$
Therefore, since $d_0 =d \neq 0$, we conclude that  $d_n\neq0$ for each $n=0,1,\ldots$ and so 
$(\phi,\psi)$ is a $x-$admissible pair.
Consequently, by (\ref{uk-is-regular}) for $k=1$, $(P_n ^{[1]})_{n\geq 0}$ is the monic OPS with respect to ${\bf u}^{[1]}$. 
This proves the last statement in the lemma for $k=1$.
Since $d_n ^{[k]}=a^{[k]}\gamma_n +\alpha_n d^{[k]}$ ($n,k=0,1,\ldots$), it is easy to see,  using (\ref{solEqSy1})--(\ref{solEqSy2}), 
that $d_n ^{[k]} =d_{n+2k} $ for all $n,k=0,1,\ldots$. Thus the last sentence in the lemma follows from (\ref{uk-is-regular}).
\end{proof}

\section{Main results: regularity conditions}\label{S-main}

In this section we state our main results, giving the analogues of Theorems B and C for OPS on lattices. 
Thus, once a lattice as in \eqref{xs-def} is fixed, we state necessary and sufficient conditions so that 
a functional ${\bf u} \in \mathcal{P}^*$ satisfying \eqref{NUL-Pearson} is regular. 
Furthermore, the monic OPS with respect to ${\bf u}$ is described. 

\subsection{The lattice $x(s)=\mathfrak{c}_1 q^{-s} +\mathfrak{c}_2 q^s +\mathfrak{c}_3$} 
We start  by considering a lattice  $x(s)$ with $q\neq1$, so that $x(s)$ is a $q-$quadratic or a $q-$linear lattice.

\begin{theorem}\label{main-Thm1}
Consider the lattice
$$
x(s)=\mathfrak{c}_1 q^{-s} +\mathfrak{c}_2 q^s +\mathfrak{c}_3\,. 
$$
Let ${\bf u}\in\mathcal{P}^*\setminus\{{\bf 0}\}$ and suppose that there exist $(\phi,\psi)\in\mathcal{P}_2\times\mathcal{P}_1\setminus\{(0,0)\}$ 
such that 
\begin{equation}\label{NUL-PearsonMainThm1}
\mathbf{D}_{x}(\phi{\bf u})=\mathbf{S}_{x}(\psi{\bf u})\;.
\end{equation}
Set $\phi(z)=az^2+bz+c$ and $\psi(z)=dz+e$ ($a,b,c,d,e \in \mathbb{C}$),
and let $d_n$ and $e_n$ be defined by \eqref{dnen-Prop}, and $\phi^{[n]}$ and $\psi^{[n]}$ be given by \eqref{psi-explicit}--\eqref{phi-explicit}. 
Then,  ${\bf u}$ is regular if and only if $(\phi,\psi)$ is a $x-$admissible pair and $\psi^{[n]}\nmid\phi^{[n]}$ for each $n=0,1,\ldots$; 
that is to say, ${\bf u}$ is regular  if and only if the following conditions hold:
\begin{equation}\label{le1a}
d_n\neq0\;,\quad \phi^{[n]}\left(\mathfrak{c}_3 -\frac{e_n}{d_{2n}}\right)\neq0\quad(n=0,1,\ldots)\,.
\end{equation}
Under such conditions, the monic OPS $(P_n)_{n\geq 0}$ with respect to ${\bf u}$ satisfies 
\begin{equation}\label{ttrr-Dx}
P_{n+1}(z)=(z-B_n)P_n(z)-C_nP_{n-1}(z) \quad(n=0,1,\ldots),
\end{equation}
with $P_{-1}(z)=0$, where the recurrence coefficients are given by
\begin{align}
B_n & =\mathfrak{c}_3+ \frac{\gamma_n e_{n-1}}{d_{2n-2}}
-\frac{\gamma_{n+1}e_n}{d_{2n}}\, ,\label{Bn-Dx} \\
C_{n+1} & =-\frac{\gamma_{n+1}d_{n-1}}{d_{2n-1}d_{2n+1}}\phi^{[n]}\left(\mathfrak{c}_3 -\frac{e_{n}}{d_{2n}}\right)\label{Cn-Dx}
\end{align}
$(n=0,1,\ldots)$.
Moreover, the following functional Rodrigues formula holds:
\begin{align}\label{RodThemMain}
P_n {\bf u} = k_n\mathbf{D}_x^n {\bf u}^{[n]}\;, \quad
k_n := (-\alpha)^{-n} \prod_{j=1} ^n d_{n+j-2} ^{-1} 
\quad(n=0,1,\ldots)\,.
\end{align}
\end{theorem}

\begin{proof}
Suppose that ${\bf u}$ is regular. 
By Lemma \ref{x-admissible}, $(\phi,\psi)$ is a $x-$admissible pair, meaning that the first condition in \eqref{le1a} holds, and
$\big(P_j^{[n]}\big)_{j\geq0}$ is the monic OPS with respect to ${\bf u}^{[n]}$, for each fixed $n$.
Write the TTRR for $\big(P_j^{[n]}\big)_{j\geq0}$ as
\begin{equation}\label{Dx-ttrrC1}
P_{j+1}^{[n]}(z)=(z-B_j^{[n]})P_j^{[n]}(z)-C_j^{[n]}P_{j-1}^{[n]}(z)\quad  (j=0,1,\ldots)
\end{equation}
($P_{-1}^{[n]}(z):=0$), being $B_{j}^{[n]}\in\mathbb{C}$
and $C_{j+1}^{[n]}\in\mathbb{C}\setminus\{0\}$ for each $j=0,1,\ldots$.
Let us compute $C_{1}^{[n]}$. 
We first show that (for $n=0$) the coefficient $C_1\equiv C_{1}^{[0]}$,
appearing in the TTRR for $(P_j)_{j\geq0}$,
is given by
\begin{equation}\label{Dx-g1}
C_1=-\frac{1}{d\alpha +a}\,\phi\left(-\frac{e}{d}\right) =-\frac{1}{d_1}\,\phi\left(\mathfrak{c}_3 -\frac{e_0}{d_0}\right)\; .
\end{equation}
Indeed, taking $k=0$ and $k=1$ in the relation
$\langle{\bf D}_{x}(\phi{\bf u}),z^k\rangle=\langle {\bf S}_x(\psi{\bf u}),z^k\rangle$,
we obtain $0=du_1+eu_0$ and
$au_2+bu_1+cu_0=-d\alpha u_2-(e\alpha +d\beta)u_1 -e\beta u_0$, where
$u_k:=\langle{\bf u},z^k\rangle$ ($k=0,1,\ldots$). Therefore,
\begin{equation}
u_1=-\frac{e}{d}u_0 \; , \quad u_2=-\frac{1}{d\alpha+a}\left[-(b+e\alpha)\frac{e}{d}+c \right] u_0 \; .\label{Dx-u1}
\end{equation}
On the other hand, since $P_1(z)=z-B_0^{[0]}=z-u_1/u_0$, we also have
\begin{equation}
C_1=\frac{\langle{\bf u},P_1^2\rangle}{u_0}=\frac{u_2u_0-u_1^2}{u_0^2}
=\frac{u_2}{u_0}-\left(\frac{u_1}{u_0}\right)^2 \; .
\label{Dx-u2}
\end{equation}
Substituting $u_1$ and $u_2$ given by (\ref{Dx-u1}) into (\ref{Dx-u2}) yields (\ref{Dx-g1}).
Since equation (\ref{uk-funcEq}) fulfilled by ${\bf u}^{[n]}$ is of the same type as (\ref{NUL-PearsonMainThm1}) fulfilled by ${\bf u}$, 
with the polynomials $\phi^{[n]}$ and $\psi^{[n]}$ in (\ref{uk-funcEq}) playing the roles of $\phi$ and $\psi$ in (\ref{NUL-PearsonMainThm1}),  
 $C_{1}^{[n]}$ may be obtained replacing in (\ref{Dx-g1}) $\phi$ and $\psi(z)=dz+e$ by $\phi ^{[n]}$ and $\psi^{[n]}(z)=d_{2n}(z-\mathfrak{c}_3)+e_n$, respectively. 
 Hence,
\begin{equation}\label{g1n}
C_{1}^{[n]}=-\frac{1}{d_{2n}\alpha +a^{[n]}}\phi ^{[n]} \left(\mathfrak{c}_3 -\frac{e_n}{d_{2n}}\right)
=-\frac{1}{d_{2n+1}}\phi ^{[n]}\left(\mathfrak{c}_3 -\frac{e_n}{d_{2n}}\right) \;.
\end{equation}
Since ${\bf u}^{[n]}$ is regular, then $C_{1}^{[n]}\neq0$, and so 
the second condition in \eqref{le1a} holds.

Conversely, suppose that conditions \eqref{le1a} hold.
Define a sequence of polynomials, $(P_n)_{n\geq 0}$, by the TTRR \eqref{ttrr-Dx}--\eqref{Cn-Dx}, with $P_{-1}(z):=0$. 
According to the hypothesis \eqref{le1a}, $C_{n+1} \neq 0$ for each $n=0,1,\ldots$. Therefore, Favard's theorem ensures that $(P_n)_{n\geq 0}$ is a monic OPS. 
To prove that ${\bf u}$ is regular we will show that $(P_n)_{n\geq 0}$ is the monic OPS with respect to $\textbf{u}$. 
For that we only need to prove that $$ u_0 \neq 0\;,\quad\left\langle \textbf{u},P_n \right\rangle =0\quad (n= 1,2,\ldots)\,.$$ 
We start by showing that  $ u_0 \neq 0$. 
Indeed, suppose that  $u_0 =0$. 
Since (\ref{NUL-PearsonMainThm1}) holds, then $\left\langle \mathbf{D}_x (\phi \textbf{u})-\mathbf{S}_x (\psi \textbf{u}), z^n \right\rangle =0$ ($n=0,1,\ldots$). 
This implies 
\begin{align}
d_n u_{n+1} +\sum_{j=0}^{n} a_{n,j} u_j =0 \quad (n=0,1,\ldots)  \label{momentseq}
\end{align}
for some complex numbers $a_{n,j}$ ($j=0,1,\ldots,n$).
Since $u_0=0$ and $d_n\neq0$ for all $n=0,1,\ldots$, from (\ref{momentseq})  we deduce $u_n =0$ for each $n=0,1,\ldots$.
This implies $\textbf{u} = {\bf 0}$, contrary to the hypothesis. Therefore $u_0 \neq 0$.
Next, notice that $P_n (z) =k_n  R_n (z)$ ($n=0,1,\ldots$), where $R_n$ is defined by \eqref{Rn-Prop1} and
$k_n ^{-1}:=(-\alpha)^n \prod_{j=1} ^{n} d_{n+j-2}$. 
Therefore, by the Rodrigues formula \eqref{roformula} given in Theorem \ref{uk-is-classical}, we obtain
\begin{align*}
\langle \textbf{u}, P_n \rangle &= k_n  \langle \textbf{u}, R_n \rangle =k_n  \langle R_n \textbf{u}, 1 \rangle 
= k_n  \big\langle \mathrm{D}_x ^n \textbf{u} ^{[n]}, 1 \big\rangle= (-1)^n k_n  \big\langle \textbf{u}^{[n]}, \mathrm{D}_x ^n ~1 \big\rangle =0  
\end{align*}
for each $n=1,2,\ldots$. Hence ${\bf u}$ is regular.
The remaining statements in the theorem follow from Theorem \ref{uk-is-classical}.
\end{proof}

\begin{remark}
It is important to highlight that the Rodrigues formula holds for a functional ${\bf u}$ satisfying a  $x-$GP functional equation, even if ${\bf u}$ is not regular, provided that
the pair $(\phi,\psi)$ appearing in that $x-$GP functional equation forms a $x-$admissible pair.
This fact is a consequence of Theorem \ref{uk-is-classical}.
\end{remark}
 
\subsection{The lattice $x(s)=\mathfrak{c}_4s^2 +\mathfrak{c}_5 s +\mathfrak{c}_6$} 
We consider now a lattice $x(s)$ for $q=1$, so that $x(s)$ is a quadratic or a linear lattice.
Recall that, here, $\mathfrak{c}_4=4\beta$. We just give the results for this situation, since the techniques and computations are very similar to the ones presented for the case $q\neq1$. For the lattice $x(s)=\mathfrak{c}_4s^2 +\mathfrak{c}_5 s +\mathfrak{c}_6$ the system of difference equations \eqref{EqSy1}--\eqref{EqSy4} becomes
\begin{align*}
a^{[n+1]}&=a^{[n]} , \\ 
d^{[n+1]}&=2a^{[n]}+d^{[n]}, \\ 
b^{[n+1]}&=b^{[n]} +6\beta (a^{[n]} +d^{[n]}), \\
e^{[n+1]}&=e^{[n]} +b^{[n]} +\beta(2a^{[n]}+3d^{[n]}),\\
c^{[n+1]}&=c^{[n]}+\beta (b^{[n]} +2e^{[n]}) +\beta^2 d^{[n]}+ \Big(\beta^2 -4\beta \mathfrak{c}_6+\frac{\mathfrak{c}_5 ^2}{4} \Big)\big( a^{[n]} +d^{[n]}\big) .
\end{align*}
with the initial conditions $a^{[0]}=a$, $b^{[0]}=b$, $c^{[0]}=c$, $d^{[0]}=d$ and $e^{[0]}=e$. 
The solution of this system is 
\begin{align*}
&a^{[n]}=a\;,\quad b^{[n]}=b+6\beta n(an+d)\;,\quad d^{[n]}=2an+d \;, \\
&e^{[n]}= bn+e+2d\beta n^2 +\beta n^2 (2an+d)\;, \\
&c^{[n]}=\phi(\beta n^2)+2\beta n\psi(\beta n^2)-n\left( 4\beta \mathfrak{c}_6 -\frac{\mathfrak{c}_5 ^2}{4}\right)(an+d) 
\end{align*}
($n=0,1,\ldots$).
Thus, applying a limiting process (as $q\to1$) on Theorem \ref{main-Thm1}, we may infer and then to prove rigorously the following 

\begin{theorem}\label{main-thm-quadratic}
Consider the lattice  $$x(s)=\mathfrak{c}_4 s^2 +\mathfrak{c}_5 s +\mathfrak{c}_6\,.$$ 
Let ${\bf u}\in\mathcal{P}^*\setminus\{{\bf 0}\}$ and suppose that there exist $(\phi,\psi)\in\mathcal{P}_2\times\mathcal{P}_1\setminus\{(0,0)\}$ such that 
\begin{equation}\label{NUL-PearsonMainThm2}
\mathbf{D}_{x}(\phi{\bf u})=\mathbf{S}_{x}(\psi{\bf u})\;.
\end{equation}
Set $\phi(z):=az^2+bz+c$ and $\psi(z):=dz+e$. 
Then {\bf u} is regular if and only if
\begin{equation}\label{le1a-quadratic}
d_n\neq0\;,\quad \phi^{[n]}\left(-\beta n^2 -\frac{e_n}{d_{2n}}\right)\neq0 \quad (n=0,1,\ldots)\,, 
\end{equation}
where $d_n :=an+d$, $e_n:=bn+e+2\beta dn^2$, and 
$$\phi ^{[n]}(z)=az^2 +(b+6\beta nd_n)z+ \phi(\beta n^2)+2\beta n\psi(\beta n^2)-\frac{n}{4}\left( 16\beta \mathfrak{c}_6 -\mathfrak{c}_5 ^2\right)d_n\,.$$ 
Under these conditions, the monic OPS $(P_n)_{n\geq 0}$ with respect to ${\bf u}$ satisfies \eqref{ttrr-Dx} with 
\begin{align}
B_n &= \frac{ne_{n-1}}{d_{2n-2}} -\frac{(n+1)e_n}{d_{2n}} -2\beta n(n-1),   \label{Bn-quadratic}\\
C_{n+1} &=-\frac{(n+1)d_{n-1}}{d_{2n-1}d_{2n+1}}\phi ^{[n]}\left(-\beta n^2 -\frac{e_n}{d_{2n}}  \right)\label{Cn-quadratic}
\end{align} 
$(n=0,1,\ldots)$.
Moreover, the following functional Rodrigues formula holds:
\begin{align}\label{RodThemMain2}
P_n {\bf u} = k_n\mathbf{D}_x^n {\bf u}^{[n]}\;, \quad
k_n := (-1)^{n} \prod_{j=1} ^n d_{n+j-2} ^{-1} 
\quad(n=0,1,\ldots)\,.
\end{align}
\end{theorem}

\begin{remark}
For $x(s)=\mathfrak{c}_6$ we immediately recover Theorem B.
\end{remark}

\section{Two examples: Racah and Askey-Wilson polynomials}\label{secAW}

In this section we derive the recurrence coefficients for the Racah and the Askey-Wilson polynomials,
using as departure point their associated $x-$GP functional equations.
It is worth mentioning that Theorem \ref{main-Thm1} allow us also to answer positively to a conjecture posed by Ismail \cite[Conjecture 24.7.8]{I2005}.
Since the proof of this conjecture is rather technique and involves long computations, it will be presented in a different work \cite{KDPconj}.

\subsection{The Racah polynomials}
Consider the quadratic lattice $$x(s)=s(s+a+b+1)\,,$$ where $a,b,c,d\in\mathbb{C}$.  
Let ${\bf u}\in\mathcal{P}^*$ be a functional satisfying \eqref{NUL-PearsonMainThm2}, where 
\begin{align*}
\phi(z)&:=2z^2 +[(a+b+2c+3)d+c(a-b+3)+2(a+b+ab+2)   ]z \\ &\quad +(1+a)(1+d)(a+b+1)(b+c+1) ,\\ 
\psi(z)&:=2(d+c+2)z+2(1+a)(1+d)(b+c+1).
\end{align*}
According to \eqref{le1a-quadratic} in Theorem \ref{main-thm-quadratic}, ${\bf u}$ is regular if and only if 
$$
\{a,c,d,d+c-1,b+c,c+d-a,d-b\}\cap\mathbb{Z}^-=\emptyset\;,
$$
where $\mathbb{Z}^-:=\{-1,-2,\ldots\}$.
Under these conditions, by \eqref{Bn-quadratic}--\eqref{Cn-quadratic}, the recurrence coefficients for the monic OPS $(P_n)_{n\geq 0}$ with respect to ${\bf u}$
are given by
\begin{align*}
B_n & =-\frac{(n+a+1)(n+d+1)(n+b+c+1)(n+d+c+1)}{(2n+d+c+1)(2n+d+c+2)} \\
&\qquad -\frac{n(n+c)(n+d+c-a)(n+d-b)}{(2n+d+c)(2n+d+c+1)} ,\\
C_{n+1}&= (n+1)(n+a+1)(n+c+1)(n+d+1)\\
&\qquad\times\frac{(n+d+c+1)(n+b+c+1)(n+c+d-a+1)(n+d-b+1)}{(2n+d+c+1)(2n+d+c+2)^2(2n+d+c+3)} 
\end{align*}
($n=0,1,\ldots$). Therefore,
$$
P_n(z)=R_n (z;d,c,a,b)\quad (n=0,1,\ldots),
$$ 
$(R_n(.;d,c,a,b))_{n\geq 0}$ being the sequence of the monic Racah polynomials \cite[p.190]{KLS2010}. 

\subsection{The Askey-Wilson polynomials}
Consider the $q-$quadratic lattice 
$$x(s)=\mathfrak{c}_1 q^{-s}+\mathfrak{c}_2 q^{s}+\mathfrak{c}_3\,.$$ 
Let ${\bf u}$ be a linear functional on $\mathcal{P}$ satisfying \eqref{NUL-PearsonMainThm1}, where $\phi$ and $\psi$ are given by
\begin{align*}
\phi (z)&= 2(1+abcd)(z-\mathfrak{c}_3)^2 \\
&\quad -2\sqrt{\mathfrak{c}_1\mathfrak{c}_2}(a+b+c+d+abc+abd+acd+bcd)(z-\mathfrak{c}_3) \\
&\quad+4\mathfrak{c}_1\mathfrak{c}_2(ab+ac+ad+bc+bd+cd-abcd-1), \\
\psi (z)&= \frac{4q^{1/2}}{q-1}\Big((abcd-1)(z-\mathfrak{c}_3) \\
&\quad+\sqrt{\mathfrak{c}_1\mathfrak{c}_2}(a+b+c+d-abc-abd-acd-bcd) \Big) ,
\end{align*}
with $a,b,c,d\in \mathbb{C}$.
The parameter $d_n$ given by \eqref{dnen-Prop} reads as
$$
d_n=-\left(\mbox{$\frac{q^{1/2}-q^{-1/2}}{2}$}\right)^{-1}q^{-n/2}\big(1-abcd q^n\big) \,. 
$$
By Theorem \ref{main-Thm1}, ${\bf u}$ is regular if and only if
$$
\begin{array}l
\mathfrak{c}_1\mathfrak{c}_2(1-abcdq^n)(1-abq^n)(1-acq^n) \\ [0.5em]
\qquad\quad\times(1-adq^n)(1-bcq^n)(1-bdq^n)(1-cdq^n) \neq 0\quad(n=0,1,\ldots)\,.
\end{array}
$$
Under these conditions, using formulas \eqref{Bn-Dx}--\eqref{Cn-Dx} in Theorem \ref{main-Thm1}, we obtain the recurrence coefficients for the monic OPS $(P_n)_{n\geq 0}$ with respect to ${\bf u}$:
\begin{align*}
B_n =\mathfrak{c}_3+2 \sqrt{\mathfrak{c}_1\mathfrak{c}_2}&\left[ a+\frac{1}{a}-\frac{(1-abq^n)(1-acq^n)(1-adq^n)(1-abcdq^{n-1})}{a(1-abcdq^{2n-1})(1-abcdq^{2n})}\right.   \\
&\quad \quad \quad \left. -\frac{a(1-q^n)(1-bcq^{n-1})(1-bdq^{n-1})(1-cdq^{n-1})}{(1-abcdq^{2n-1})(1-abcdq^{2n-2})}\right]
\end{align*}
(if $a=0$, define $B_n$ by continuity, taking $a\to0$ in the preceding expression), and
\begin{align*}
C_{n+1}&=\mathfrak{c}_1\mathfrak{c}_2(1-q^{n+1})(1-abcdq^{n-1}) \\
&\quad\times \frac{(1-abq^n)(1-acq^n)(1-adq^n)(1-bcq^n)(1-bdq^n)(1-cdq^n)}{(1-abcdq^{2n-1})(1-abcdq^{2n})^2 (1-abcdq^{2n+1})}
\end{align*}
($n=0,1,2,\ldots$).
Hence 
$$P_n (z)= 2^n(\mathfrak{c}_1\mathfrak{c}_2)^{n/2} Q_n \left(\frac{z-\mathfrak{c}_3}{2\sqrt{\mathfrak{c}_1\mathfrak{c}_2}};a,b,c,d|q  \right)\quad (n=0,1,\ldots),$$
where $(Q_n(\cdot;a,b,c,d|q))_{n\geq0}$ is the monic OPS of the Askey-Wilson polynomials (see \cite[(14.1.5)]{KLS2010}).

\section*{Acknowledgements }
This work is supported by the Centre for Mathematics of the University of Coimbra -- UID/MAT/00324/2019, funded by the Portuguese Government through FCT/MEC and co-funded by the European Regional Development Fund through the Partnership Agreement PT2020. DM is also supported by the FCT grant PD/BD/135295/2017.

\end{document}